\documentclass{scrartcl}

\usepackage{lmodern}

\usepackage[T1]{fontenc}
\usepackage[utf8]{inputenc}
\usepackage[english]{babel}

\usepackage{amssymb}
\usepackage{amsmath}
\usepackage{amsthm}

\usepackage{hyperref}
\usepackage[capitalise]{cleveref}

\theoremstyle{plain}
\newtheorem{theorem}{Theorem}[section]
\newtheorem*{theorem*}{Theorem}
\newtheorem{proposition}[theorem]{Proposition}
\newtheorem{lemma}[theorem]{Lemma}
\newtheorem{corollary}[theorem]{Corollary}

\theoremstyle{definition}
\newtheorem{definition}[theorem]{Definition}

\newtheorem{fact}[theorem]{Fact}
\newtheorem{observation}[theorem]{Observation}
\newtheorem{conjecture}[theorem]{Conjecture}

\theoremstyle{remark}
\newtheorem{remark}[theorem]{Remark}
\newtheorem{question}[theorem]{Question}
\newtheorem*{claim}{Claim}

\newcommand{\Aut}{\mathrm{Aut}}

\newcommand{\Th}{\mathrm{Th}}

\newcommand{\tor}{\mathrm{tor}}

\newcommand{\bbZ}{\mathbb{Z}}
\newcommand{\bbP}{\mathbb{P}}
\newcommand{\bbF}{\mathbb{F}}
\newcommand{\bbN}{\mathbb{N}}
\newcommand{\bbA}{\mathbb{A}}

\newcommand{\calZ}{\mathcal{Z}}
\newcommand{\calL}{\mathcal{L}}
\newcommand{\calI}{\mathcal{I}}
\newcommand{\calS}{\mathcal{S}}
\newcommand{\calG}{\mathcal{G}}
\newcommand{\calF}{\mathcal{F}}
\newcommand{\calN}{\mathcal{N}}
\newcommand{\calP}{\mathcal{P}}
\newcommand{\calU}{\mathcal{U}}
\newcommand{\calA}{\mathcal{A}}
\newcommand{\calM}{\mathcal{M}}
\newcommand{\calT}{\mathcal{T}}

\newcommand{\Cen}{\mathrm{Cen}}

\newcommand{\Div}{\mathrm{Div}}
\newcommand{\Ind}{\mathrm{Ind}}

\newcommand{\Sym}{\mathrm{Sym}}

\newcommand{\dprk}{\textrm{dp-rk}}
\newcommand{\lcm}{\mathrm{lcm}}

\renewcommand{\phi}{\varphi}

\begin{document}

\title{Dp-minimal profinite groups and \\ valuations on the integers}
\author{Tim Clausen\footnote{The author is partially supported by the Deutsche Forschungsgemeinschaft (DFG, German Research Foundation) under Germany’s Excellence Strategy EXC 2044–390685587, Mathematics M\"unster: Dynamics-Geometry-Structure. }}
\maketitle

\begin{abstract}
We study dp-minimal infinite profinite groups that are equipped with a uniformly definable fundamental system of open subgroups. We show that these groups have an open subgroup $A$ such that either $A$ is a direct product of countably many copies of $\bbF_p$ for some prime $p$, or $A$ is of the form $A \cong \prod_p \bbZ_p^{\alpha_p} \times A_p$ where $\alpha_p < \omega$ and $A_p$ is a finite abelian $p$-group for each prime $p$. Moreover, we show that if $A$ is of this form, then there is a fundamental system of open subgroups such that the expansion of $A$ by this family of subgroups is dp-minimal. Our main ingredient is a quantifier elimination result for a class of valued abelian groups. We also apply it to $(\bbZ,+)$ and we show that if we expand $(\bbZ,+)$ by any chain of subgroups $(B_i)_{i<\omega}$, we obtain a dp-minimal structure. This structure is distal if and only if the size of the quotients $B_i/B_{i+1}$ is bounded.
\end{abstract}

\section{Introduction}
A profinite group $G$ together with a fundamental system $\{ K_i : i \in I \}$ of open subgroups  can be viewed as a two-sorted structure $(G,I)$ in the two-sorted language $\calL_\text{prof}$. In these structures the fundamental system of open subgroups is definable. Since a fundamental system of open subgroups is a neighborhood basis at the identity, this implies that the topology on $G$ is definable.

These structures have been studied by Macpherson and Tent in \cite{macpherson-tent}. They mainly considered \emph{full} profinite groups, i.e. profinite groups $G$ where the family $\{K_i : i\in I \}$ consists of all open subgroups. Their main result states that a full profinite group $(G,I)$ is NIP if and only if it is NTP$_2$ if and only if it is virtually a finite direct product of analytic pro-$p$ groups.

Since analytic pro-$p$ groups can be described as products of copies of $\bbZ_p$ with a twisted multiplication, profinite NIP groups are composed of ``one-dimensional'' profinite NIP groups. In the setting of full profinite groups the combinatorial structure of the lattice of open subgroups is visible in the model theoretic structure. This plays an important role in the classification.

Without the fullness assumption, only a portion of this lattice is visible. In general the family $\{K_i : i \in I \}$ could simply consist of a chain of open subgroups. In this more general setting, we will restrict ourselves to the ``one-dimensional'', i.e. dp-minimal case. A profinite group $(G,I)$ is dp-minimal if it has NIP and is dp-minimal in the group sort. We prove the following classification result:

\begin{theorem*}
Let $(G,I)$ be a dp-minimal profinite group. Then $G$ has an open abelian subgroup $A$ such that either $A$ is a direct product of countably many copies of $\bbF_p$ for some prime $p$, or $A$ is isomorphic to $\prod_{p} \bbZ_p^{\alpha_p} \times A_p$ where $\alpha_p < \omega$ and $A_p$ is a finite abelian $p$-group for each prime $p$. Moreover, every abelian profinite group $A$ of the above form admits a fundamental system of open subgroups such that the corresponding $\calL_\text{prof}$-structure is dp-minimal.
\end{theorem*}

The main ingredient of this theorem is a quantifier elimination result which is also applicable in other settings. We apply it to this situation:
Consider the structure $(\bbZ,+)$. If we expand it by the full lattice of subgroups, then the expanded structure interprets Peano Arithmetic and hence is not tame in any sense.
However, if we only name a chain in this lattice, we obtain a tame structure. 
A chain of subgroups $ \bbZ = B_0 > B_1 > \dots$ is the same as a valuation $v: \bbZ \rightarrow \omega \cup \{ \infty \}$ defined by
\[ v(a) = \max\{ i : a \in B_i \}. \]

\begin{theorem*}
  Let $(B_i)_{i<\omega}$ be a strictly descending chain of subgroups of $\bbZ$, $B_0 = \bbZ$, and let $v: \bbZ \rightarrow \omega \cup \{\infty\}$ be the valuation defined by
  \[ v(x) = \max\{ i : x \in B_i \}. \]
  Then $(\bbZ,0,1,+,v)$ is dp-minimal. Moreover, $(\bbZ,0,1,+,v)$ is distal if and only if the size of the quotients $B_i/B_{i+1}$ is bounded.
\end{theorem*}

This stays true if we expand the value sort by unary predicates and monotone binary relations.
There has been recent interest in dp-minimal expansions of $(\bbZ,+)$ (e.g. \cite{alouf-delbee}, \cite{tran-walsberg}, \cite{alouf}, and \cite{walsberg}).
Alouf and d'Elb\'ee showed in \cite{alouf-delbee} that if $p$ is a prime and $v_p$ denotes the $p$-adic valuation, then $(\bbZ, 0, 1, + , v_p)$ is a minimal expansion of $(\bbZ,0,1,+)$ in the sense that there are no proper intermediate expansions. We show that this does not hold true for all valuations and we conjecture that the $p$-adic valuations are essentially the only examples with this property among valuations $v$ such that $(\bbZ,0,1,+,v)$ is distal. 

The proof of the classification theorem for dp-minimal profinite groups consists of three parts: We analyze the algebraic structure of dp-minimal profinite groups in \cref{ch:algebraic_structure}. This will imply the first part of the theorem.
It then remains to show that these groups appear as dp-minimal profinite groups. This is done in \cref{ch:valued_groups}. The case where the group is given by an $\bbF_p$-vector space has already been done by Maalouf in \cite{maalouf}. We explain this result in \cref{sec:maalouf}.
The remaining case is handled by a quantifier elimination result (see \cref{sec:qe_result}). This quantifier elimination result allows us to show that a certain class of profinite groups as 
$\calL_\text{prof}$-structures is dp-minimal  (\cref{thm:dp-minimality}) and we are able to characterize distality in this class (\cref{thm:distality}).

We will also apply the quantifier elimination result to valuations on $(\bbZ,+)$. This will be done in \cref{ch:integers} where we discuss the second theorem and its consequences for the study of dp-minimal expansions of $(\bbZ,+,0,1)$. We also show that the $p$-adic valuations have a limit theory (\cref{prop:limit_theory}) and we consider expansions of $(\bbZ,+,0,1)$ given by multiple valuations.

\cref{ch:further} contains a few results which are related to dp-minimal profinite groups. We show that our main result implies some structural consequences for uniformly definable families of finite index subgroups in dp-minimal groups (\cref{dp-min_uniform_family}). Jarden and Lubotzky \cite{jarden-lubotzky} showed that two elementarily equivalent profinite groups are isomorphic if one of them is finitely generated. This was generalized to strongly complete profinite groups by Helbig \cite{helbig}. We will give an alternative proof for these results in \cref{sec:homogeneity}. Finally, we prove a result about uniformly definable families of normal subgroups in NTP$_2$ groups (\cref{ntp2_nip_transfer})\footnote{Thanks to Pierre Simon for bringing this question to my attention.}: If such a family is closed under finite intersections, then it must be defined by an NIP formula.

\subsection*{Acknowledgments}
I would like to thank my supervisor, Katrin Tent, for her help and support during the last years and for giving me the opportunity to work on these exciting topics.
I would also like to thank Pierre Simon for interesting and useful discussions and valuable suggestions while I visited UC Berkeley.

\section{Preliminaries}
We assume that the reader is familiar with both profinite groups and model theory. We will give a quick overview about the notions and tools that are used to prove the main result.

\subsection{Profinite groups}
A topological group is \emph{profinite} if it is the inverse limit of an inverse system of (discrete) finite groups.
This condition is equivalent to the group being Hausdorff, compact, and totally disconnected.
If $G$ is a profinite group, then
\[ G \cong \varprojlim G/N \]
where $N$ ranges over all open normal subgroups.

The open subgroups generate the topology on $G$, i.e. every open set is a union of cosets of open subgroups.
A \emph{fundamental system of open subgroups} is a family $\calF$ consisting of open subgroups which generate the topology on $G$.
Equivalently, every open subgroup of $G$ contains a subgroup in $\calF$. If $\calP$ is a property of groups, we will say that $G$ is virtually $\calP$ if $G$ has an open subgroup $H$ which satisfies $\calP$.

We will use a number of results about the structure of abelian profinite groups. Recall that a profinite group is pro-$p$ if it is the inverse limit of finite $p$-groups.
A free abelian pro-$p$ group is a direct product of copies of $\bbZ_p$. 

\begin{proposition}[Theorem 4.3.4 of \cite{ribes-zalesskii}]\label{prop:sylow_structure}
Let $p$ be a prime.
\begin{enumerate}
  \item[(a)] If $G$ is a torsion free pro-$p$ abelian group, then $G$ is a free abelian pro-$p$ group.
  \item[(b)] Let $G$ be a finitely generated pro-$p$ abelian group. Then the torsion subgroup $\tor(G)$ is finite and
  \[ G \cong F \oplus \tor(G) \]
  where $F$ is a free pro-$p$ abelian group of finite rank.
\end{enumerate}
\end{proposition}

\begin{proposition}[Corollary 4.3.9 of \cite{ribes-zalesskii}]\label{prop:torsion_structure}
Let $G$ be a torsion profinite abelian group. Then there is a finite set of primes $\pi$ and a natural number $e$ such that
\[ G \cong \prod_{p \in \pi}( \prod_{i=1}^e (\prod_{m(i,p)}C_{p^i} )) \]
where each $m(i,p)$ is a cardinal and each $C_{p^i}$ is the cyclic group of order $p^i$.
In particular, $G$ is of finite exponent.
\end{proposition}

\begin{proposition}[Proposition 1.13 and Proposition 1.14 of \cite{dixon}]\label{prop:frattini}
Let $G$ be a pro-$p$ group. Then $G$ is (topologically) finitely generated if and only if the Frattini subgroup $\Phi(G) = \overline{G^p[G,G]}$ is open in $G$.
\end{proposition}

\begin{proposition}\label{p:abelian_structure}
  Let $A$ be an abelian profinite group. Then $nA \leq A$ is an open subgroup for all $n \geq 1$  if and only if
  \[ A \cong \prod_{p} \bbZ_p^{\alpha_p} \times A_p \]
  where $\alpha_p < \omega$ and $A_p$ is a finite abelian $p$-group for each prime $p$.
\end{proposition}
\begin{proof}
  An abelian profinite group is the direct product of its $p$-Sylow subgroups. Let $P$ be a $p$-Sylow subgroup of $A$. If $pP \leq P$ has finite index, then $P$ is finitely generated by \cref{prop:frattini}. Then by \cref{prop:sylow_structure} the $p$-Sylow subgroup $P$ has the desired form.
\end{proof}

We will also need the following result by Zelmanov:
\begin{theorem}[Theorem 2 of \cite{zelmanov}] \label{abelian_subgroup_thm} Every infinite compact group has an infinite abelian subgroup.
\end{theorem}

We will view profinite groups as two-sorted structures in the following language which was introduced in \cite{macpherson-tent}:
\begin{definition} \label{def:Lprof}
$\calL_\text{prof}$ is a two-sorted language containing the group sort $\calG$ and the index sort $\calI$. The language $\calL_\text{prof}$ then consists of:
\begin{itemize}
\item the usual language of groups on $\calG$,
\item a binary relation $\leq$ on $\calI$, and
\item a binary relation $K \subseteq \calG \times \calI$.
\end{itemize}
\end{definition}

\begin{remark}\label{rem:Lprof}
A profinite group $G$ together with a fundamental system of open subgroups $\{ K_i : i \in I \}$ can be viewed as an $\calL_\text{prof}$ structure $(G,I)$ as follows:
\begin{itemize}
\item we set $i \leq j$ if and only if $K_i \supseteq K_j$, and
\item the relation $K$ is defined by $K(G,i) = K_i$ for all $i \in I$.
\end{itemize}
\end{remark}

\subsection{Model theoretic notions of complexity}
We will mostly work in the context of an NIP theory. We use \cite{simon} as our main reference for this section.

\subsubsection{The independence property}
An important class of model theoretic theories is the class of NIP (or dependent) theories, i.e. the class of theories which cannot code the $\in$-relation on an infinite set. This notion was introduced by Shelah.

\begin{definition}
A formula $\phi(x,y)$ has the \emph{independence property (IP)} if there are sequences $(a_i)_{i < \omega}$ and $(b_J)_{J \subseteq \omega}$ such that
\[ \models \phi(a_i, b_J ) \iff i \in J. \]
\end{definition}
We say that $\phi(x,y)$ has NIP if it does not have IP.
This notion is symmetric in the sense that a formula $\phi(x,y)$ has NIP if and only if the formula $\psi(y,x) \equiv \phi(x,y)$ has NIP (see Lemma 2.5 of \cite{simon}).

We will make use of the following characterization of IP:
\begin{lemma}[Lemma 2.7 of \cite{simon}]
A formula $\phi(x,y)$ has IP if and only if there exists an indiscernible sequence $(a_i)_{i < \omega}$ and a tuple $b$ such that
\[ \models \phi(a_i,b) \iff i \text{ is odd}. \]
\end{lemma}

We call a theory NIP if all formulas have NIP.

\begin{definition}
A subset $X \subseteq M \models T$ is \emph{externally definable} if there is a formula $\phi(x,y)$, an elementary expansion $M^*$ of $M$, and an element $b \in M^*$ such that $X = \phi(M,b)$.
\end{definition}

By a result of Shelah, naming all externally definable sets in an NIP structure preserves NIP:
\begin{theorem}[Proposition 3.23 and Corollary 3.24 of \cite{simon}] \label{thm:externally_definable}
Let $M$ be a model of an NIP theory and let $M^\text{Sh}$ be the Shelah expansion, i.e. the expansion of $M$ by all externally definable sets. Then $M^\text{Sh}$ has quantifier elimination and is NIP.
\end{theorem}

\begin{theorem}[Baldwin-Saxl, Theorem 2.13 of \cite{simon}] \label{baldwin-saxl}
Let $G$ be an NIP group and let $\{ H_i : i \in I \}$ be a family of uniformly definable subgroups of $G$. Then there is a constant $K$ such that
for any finite subset $J \subseteq I$ there is $J_0 \subseteq J$ of size $|J_0| \leq K$ such that
\[ \bigcap\{H_i : i \in J \} = \bigcap\{H_i : i \in J_0 \}. \]
\end{theorem} 

As an easy consequence we obtain:

\begin{corollary}\label{cor:subgroup_counting}
If $(G,I)$ is an NIP profinite group, then $\{K_i : i \in I \}$ can only contain finitely many subgroups of any given finite index.
\end{corollary}

By a result of Shelah, abelian subgroups of NIP groups have definable envelopes given by centralizers of definable sets:
\begin{theorem}[Proposition 2.27 of \cite{simon}] \label{def_abelian_thm} Let $G$ be an NIP group and let $X$ be a set of commuting elements. Then there is a formula $\phi(x,y)$ and a parameter $b$ (in some elementary extension $G^*$) such that $\Cen(\Cen(\phi(G^*,b)))$ is an abelian (definable) subgroup of $G^*$ and contains $X$.
\end{theorem}

\subsubsection{Dp-minimality}
NIP theories admit a notion of dimension given by dp-rank:

\begin{definition}[Definition 4.2 of \cite{simon}]
Let $p$ be a partial type over a set $A$, and let $\kappa$ be a cardinal. We define
\[ \dprk(p,A) < \kappa \]
if and only if for every family $(I_t)_{t<\kappa}$ of mutually indiscernible sequences over $A$ and $b \models p$, one of these sequences is indiscernible over $Ab$.
\end{definition}

A theory is called \emph{dp-minimal} if $\dprk(x=x,\emptyset) = 1$ where $x$ is a singleton. We call a multi-sorted theory with distinguished home-sort \emph{dp-minimal} if it is NIP and it is dp-minimal in the home-sort, i.e. $\dprk(x=x, \emptyset) = 1$ where $x$ is a singleton in the home-sort.

\begin{remark}\label{rem:externally_remark}
As a consequence of the quantifier elimination in \cref{thm:externally_definable} the Shelah expansion of a dp-minimal structure is dp-minimal.
\end{remark}

We will use the fact that definable subgroups in a dp-minimal group are always comparable in the following sense:
\begin{lemma}[Claim in Lemma 4.31 of \cite{simon}] \label{intersection_lemma} Suppose $G$ is dp-minimal and $H_1$ and $H_2$ are definable subgroups. Then $|H_1 : H_1 \cap H_2|$ or $|H_2:H_1 \cap H_2|$ is finite.
\end{lemma}

\subsubsection{Distality}
Distality is a notion introduced by Simon to describe the unstable part of an NIP theory. The general definition of distality is slightly more complicated than the definitions of NIP and dp-minimality (see Definition 2.1 in \cite{distal} or Chapter 9 in \cite{simon}). In case of a dp-minimal theory distality can be described as follows:

\begin{proposition} \label{distal_def}
  A dp-minimal theory $T$ is distal if and only if there is no infinite non-constant totally indiscernible set of singletons.
\end{proposition}
\begin{proof}
  This characterization follows from Example 2.4 and Lemma 2.10 in \cite{distal}.
\end{proof}

By Exercise 9.12 of \cite{simon} distality is preserved under going to $T^\text{eq}$:

\begin{proposition} \label{distal_eq}
If $T$ is distal, then so is $T^\mathrm{eq}$.
\end{proposition}

\subsection{Quantifier elimination}
Recall that a theory $T$ has quantifier elimination if every formula is equivalent to a quantifier free formula modulo $T$.
The proof of Theorem 3.2.5 in \cite{tent-ziegler} gives the following useful criterion for quantifier elimination:
\begin{proposition}\label{prop:qe_criterium}
Let $T$ be a theory and let $\phi(x)$ be a formula. Then $\phi(x)$ is equivalent to a quantifier free formula modulo $T$ if and only if
for all $\calM_1,\calM_2 \models T$ with common substructure $\calA$ and all $a \in \calA$ we have
\[ \calM_1 \models \phi(a) \implies \calM_2 \models \phi(a). \]
\end{proposition} 

If $T$ is a two-sorted theory and the only symbols that connect the two sorts are functions from one sort to the other, then it suffices to check quantifier elimination for very specific formulas:

\begin{lemma} \label{lem:qe}
	Let $T$ be a theory in a two-sorted language $\calL = \calL_0 \cup \calL_1 \cup \{f_j : j \in J\}$ with sorts $\calS_0$ and $\calS_1$ where $\calL_0$ is purely in the sort $\calS_0$, $\calL_1$ is purely in the sort $\calS_1$, and each $f_j$ is a function from sort $\calS_0$ to sort $\calS_1$. Suppose
	\begin{itemize}
		\item[(a)] every $\calL_1$-formula is equivalent to a quantifier free formula modulo $T$ and
		\item[(b)] every formula of the form 
		\[\exists x \in \calS_0 \bigwedge_{r \in R} \phi_r(x, \bar{y}_r, \bar{z}_r ) \]
		is equivalent to a quantifier free formula modulo $T$
		where $x$ is a singleton, $\bar{y}_r \subseteq \calS_0$, $\bar{z}_r \subseteq \calS_1$, and each $\phi_r$ is either a basic $\calL_0$-formula or is of the form $f_j(t(x,\bar{y}_r)) = z$ where $t$ is an $\calL_0$ term and $z$ is one of the variables in the tuple $\bar{z}_r$. 
	\end{itemize}
	Then $T$ eliminates quantifiers.
\end{lemma}
\begin{proof}
To show quantifier elimination it suffices to consider simple existential formulas. Consider a formula of the form
\[ \exists \gamma \in \calS_1 \bigwedge_{r \in R} \phi_r(\gamma, \bar{y_r}, \bar{z_r} ) \]
where $\gamma$ is a singleton, $\bar{y_r} \subseteq \calS_0$, $\bar{z_r} \subseteq \calS_1$, and each $\phi_r$ is a basic formula.
We may assume that $\gamma$ appears non-trivially in each formula $\phi_r$. Then each $\phi_r$ is a basic $\calL_1$-formula where the variables $\bar{y}_r$ only appear as terms of the form
\[f(t(\bar{y_r}))\]
where $f$ is a function symbol and $t$ is an $\calL_0$-term. Now the $\calS_1$-quantifier can be eliminated by (a).

Now consider a formula of the form
\[ \exists x \in \calS_0 \bigwedge_{r \in R} \phi_r(x, \bar{y_r}, \bar{z_r} ) \]
where $x$ is a singleton, $\bar{y_r} \subseteq \calS_0$, $\bar{z_r} \subseteq \calS_1$, and each $\phi_r$ is a basic formula.

Let $\tilde{R} \subseteq R$ be the set of all $r \in R$ such that $\phi_r$ is a basic $\calL_1$-formula.
If $r \in \tilde{R}$, then we may write $\phi_r$ as
\[ \phi_r \equiv \psi_r(\bar{f_r}(\bar{t_r}(x,\bar{y_r})), \bar{z_r})\]
where $\psi_r$ is a basic $\calL_1$-formula such that all variables of $\psi_r$ are in $\calS_1$.
Then $\phi_r$ is equivalent to
\[ \exists \bar{\xi} \in \calS_1 : ( \bar{\xi} = \bar{f_r}(\bar{t_r}(x,\bar{y_r})) \land \psi_r( \bar{\xi}, \bar{z_r})). \]
Now we may rewrite
\[ \exists x \in \calS_0 \bigwedge_{r \in R} \phi_r(x, \bar{y_r}, \bar{z_r} ) \]
as a formula of the form
\[ \exists (\bar{\xi_r})_{r\in \tilde{R}} \in \calS_1 : (( \bigwedge_{r \in \tilde{R}} \psi_r( \bar{\xi}, \bar{z_r})) \land ( \exists x \in \calS_0 \bigwedge_{r \in \tilde{R}} \bar{\xi_r} = \bar{f_r}(\bar{t_r}(x,\bar{y_r})) \land \bigwedge_{r \in R \setminus \tilde{R}} \phi_r(x,\bar{y_r} ))). \]
We can now eliminate the $\calS_0$-quantifier by (b) and then eliminate the $\calS_1$-quantifiers as in the first step.
\end{proof}

\section{Algebraic properties of dp-minimal profinite groups} \label{ch:algebraic_structure}
We view a profinite group $G$ together with a fundamental system of open subgroups $\{K_i : i \in I \}$ as an $\calL_\text{prof}$-structure $(G,I)$ (as in \cref{rem:Lprof}).
The aim of this chapter is to prove the first part of the main theorem: If $(G,I)$ is a dp-minimal profinite group, then $G$ has an open abelian subgroup $A$
such that either $A$ is a vector space over $\bbF_p$ for some prime $p$, or $A \cong \prod_{p} \bbZ_p^{\alpha_p} \times A_p$ where $\alpha_p < \omega$ and $A_p$ is a finite abelian $p$-group for each prime $p$.

Simon showed in \cite{dp-minimal_ordered} that all dp-minimal groups are abelian-by-finite-exponent. An example of a dp-minimal group that is not abelian-by-finite was given by Simonetta in \cite{simonetta}.

We will show that all dp-minimal profinite groups have an open abelian subgroup. We will then analyze the structure of this abelian profinite group.
For dp-minimal profinite groups the fundamental system of open subgroups can always be replaced by a chain of open subgroups:

\begin{lemma} \label{chain_lemma}
	Let $(G,I)$ be a dp-minimal profinite group. Then the subgroups
	\[ H_i := \bigcap \{ K_j : |G:K_j| \leq |G:K_i| \} \]
	are uniformly definable open subgroups and hence the topology on $G$ is generated by a definable chain of open subgroups.
\end{lemma}
\begin{proof}
  The $H_i$ are open subgroups by \cref{cor:subgroup_counting}.
	By \cref{intersection_lemma} and compactness we can find a constant $K$ such that for all $i,j$ $|K_i : K_i \cap K_j | < K$ or $|K_j : K_i \cap K_j| < K$.
	Given $i,j \in I$ we have
	\[ |G: K_i | \leq |G:K_j| \iff |K_i : K_i \cap K_j | \geq |K_j : K_i \cap K_j|. \]
	Moreover, we have $|K_i : K_i \cap K_j | < K$ or $|K_j : K_i \cap K_j| < K$. Therefore this is a definable condition and hence the subgroups $H_i$ are uniformly definable.
\end{proof}

In a dp-minimal profinite group we cannot find infinite definable subgroups of infinite index:
\begin{lemma} \label{finite_index} Let $(G,I)$ be a dp-minimal profinite group. Let $(G^*,I^*)$ be an elementary extension and let $H < G^*$ be a definable subgroup. If $G \cap H$ is infinite, then $|G^*:H|$ is finite.
\end{lemma}
\begin{proof}
	If $|K_i^*:K_i^* \cap H|$ is finite for some $i \in I$, then clearly $|G^*:H| < \infty$.
	
	Now assume $|K_i^*:K_i^*\cap H|$ is infinite for all $i \in I$. We aim to show that $|H:K_i^*\cap H|$ must be unbounded:
	Since $G \cap H$ is infinite and $\bigcap_{i \in I} K_i = 1$, $|G\cap H:K_i \cap H|$ must be unbounded. Therefore $|H:K_i^*\cap H|$ must be unbounded. This contradicts \cref{intersection_lemma}.
\end{proof}

As a consequence of Zelmanov's theorem (\cref{abelian_subgroup_thm}) and the existence of definable envelopes for abelian subgroups (\cref{def_abelian_thm}) we get that a dp-minimal profinite group must be virtually abelian:
\begin{proposition} \label{prop:virtual_abelian}
Let $(G,I)$ be a dp-minimal profinite group. Then $G$ is virtually abelian.
\end{proposition}
\begin{proof}
By \cref{abelian_subgroup_thm} $G$ has an infinite abelian subgroup $A$. By \cref{def_abelian_thm}, we can find an elementary extension $(G^*,I^*)$, a formula $\phi(x,y)$, and a parameter $b \in (G^*,I^*)$ such that $\Cen(\Cen(\phi(G^*,b)))$ is an abelian subgroup of $G^*$ and contains $A$. Therefore $\Cen(\Cen(\phi(G^*,b)))$ has finite index in $G^*$ by \cref{finite_index}. By elementarity there is some $b' \in (G,I)$ such that $\Cen(\Cen(\phi(G,b')))$ is an abelian group and has finite index in $G$. Moreover, $\Cen(\Cen(\phi(G,b')))$ is closed since it is a centralizer. Closed subgroups of finite index are open and therefore $\Cen(\Cen(\phi(G,b')))$ is an open abelian subgroup of $G$.
\end{proof}

We are now able to prove the first part of the main theorem:
\begin{theorem} \label{thm:dp-min_structure}
Let $(A,I)$ be an abelian dp-minimal profinite group. Then either $A$ is virtually a direct product of countably many copies of $\bbF_p$ for some prime $p$, or $A \cong \prod_{p} \bbZ_p^{\alpha_p} \times A_p$ where $\alpha_p < \omega$ and $A_p$ is a finite abelian $p$-group for each prime $p$.
\end{theorem}
\begin{proof}
	Consider the closed subgroup $A[n] := \{ x \in A : nx = 0 \}$. Suppose there is a minimal $n$ such that $A[n]$ is infinite.
	Then $A[n]$ has finite index in $A$ (by \cref{finite_index}) and hence is an open subgroup of $A$. Therefore we may assume $A = A[n]$.
	Now the minimality of $n$ and \cref{prop:torsion_structure} imply that $n$ must be prime and therefore $A$ is a direct product of copies of $\bbF_p$ (again by \cref{prop:torsion_structure}). Since $A$ admits a countable fundamental system of open subgroups, this direct product must be a direct product of countably many copies of $\bbF_p$.
	
	Now assume $A[n]$ is finite for all $n$. Then the closed subgroup $nA$ must be open in $A$ for all $n$ (by \cref{finite_index}). Now \cref{p:abelian_structure} implies the theorem.
\end{proof}

\section{Valued abelian profinite groups} \label{ch:valued_groups}

If $A$ is an abelian group and $(A_i)_{i < \omega}$ is a strictly descending chain of subgroups such that $A_0 = A$ and $\bigcap_{i<\omega}A_i = \{0\}$,
then we can define a valuation map $v : A \rightarrow \omega \cup \{\infty\}$ by setting
\[ v(x) = \max\{i : x \in A_i \}. \]
We have $v(x) = \infty$ if and only if $x = 0$, and this valuation satisfies the inequality
\[ v(x-y) \geq \min\{v(x), v(y)\} \]
where we have equality in case $v(x) \neq v(y)$.

The valued group $(A,v)$ can be seen as a two-sorted structure consisting of the group $A$, the linear order $(\omega \cup \{ \infty \}, \leq )$, and the valuation $v: A \rightarrow \omega \cup \{ \infty \}$.

Our goal is to classify dp-minimal profinite groups up to finite index. We know by \cref{chain_lemma} that the fundamental system of open subgroups can be assumed to be a chain.
Moreover, by \cref{thm:dp-min_structure} we only need to consider groups of the form
\[ \prod_{i < \omega}\bbF_p \quad \text{or} \quad \prod_{p} \bbZ_p^{\alpha_p} \times A_p \]
where $\alpha_p < \omega$ and $A_p$ is a finite abelian $p$-group for each prime $p$.

If $A$ is such a group and $\{B_i : i < \omega \}$ is a fundamental system of open subgroups which is given by a strictly descending chain,
then the above construction yields a definable valuation $v : A \rightarrow \omega \cup \{\infty\}$. Conversely, given such a valuation $v$, we can recover the fundamental system of open subgroups by setting
\[ B_i = \{ a \in A : v(a) \geq i \}. \]
Hence the valuation and the fundamental system are interdefinable.

We will show that if $A$ is of the above form, then $A$ admits a fundamental system given by a chain of open subgroups such that
the expansion of $A$ by the corresponding valuation (and hence the corresponding $\calL_\text{prof}$-structure) is dp-minimal.
If $A = \prod_{ i < \omega} \bbF_p$, this follows from results by Maalouf in \cite{maalouf} and will be explained in \cref{sec:maalouf}.

\begin{definition} \label{good_valuation}
\begin{enumerate}
\item[(a)] The subgroups $B_i = \{ a \in A : v(a) \geq i \}$ are called the \emph{$v$-balls of radius $i$}. We will also denote them by $B_i^v$ to emphasize that they correspond to the valuation $v$. 
\item[(b)] A valuation $v: A \rightarrow \omega \cup \{ \infty \}$ is \textit{good} if 
 for all $i < \omega$ the subgroup $B_i$ is of the form $B_i = nA$ for some positive integer $n$.
\end{enumerate} 
\end{definition} 

In case $A \cong  \prod_{p} \bbZ_p^{\alpha_p} \times A_p$, we will prove a quantifier elimination result for good valuations. Note that by \cref{p:abelian_structure} each such group can be equipped with a good valuation such that $\{B_i : i < \omega \}$ is a fundamental system of open subgroups. We will show the following theorem:

\begin{theorem}
Let $A \cong  \prod_{p} \bbZ_p^{\alpha_p} \times A_p$ as above and let $v$ be a good valuation. Then the structure $(A,+,v)$ is dp-minimal. Moreover, it is distal if and only if the size of the quotients $B_i / B_{i+1}$ is bounded.
\end{theorem}
This theorem will be proven in this chapter.
If $\pi$ is a set of primes, a natural number $n \geq 1$ is called a $\pi$-number if the prime decomposition of $n$ only contains primes in $\pi$. An immediate consequence of the above theorem is the following:
\begin{corollary}
Let $(\pi_i)_{i < \omega}$ be a sequence of finite non-empty disjoint sets of primes. For each $i<\omega$ fix a finite non-trivial abelian group $A_i$ such that $|A_i|$ is a $\pi_i$-number.
Set 
\[ A = \prod_{i<\omega}A_i\]
and let $v$ be the valuation defined by
\[ v((a_i)_{i<\omega}) = \min\{i : a_i \neq 0  \}. \]
Then $(A,+,v)$ is dp-minimal but not distal.
\end{corollary}
\begin{proof}
  We have $B_k^v = (\prod_{i<k}|A_i|)A$. Hence $v$ is a good valuation and the theorem applies.
\end{proof}

\subsection{Valued vector spaces} \label{sec:maalouf}
Valued vector spaces have been studies by S. Kuhlmann and F.-V. Kuhlmann in \cite{kuhlmann} and by Maalouf in \cite{maalouf}.
Set $A = \prod_{i < \omega}\bbF_p$ and let $v : A \rightarrow \omega \cup \{\infty\}$ be the valuation given by
\[  v((x_i)_{i<\omega}) = \min \{ i : x_i \neq 0 \}. \]
It follows from results by Maalouf in \cite{maalouf} that this valued abelian profinite group is dp-minimal:

\begin{proposition} \label{prop:valued_vector_space}
The valued abelian profinite group $(A,v)$ is dp-minimal.
\end{proposition}
\begin{proof}
  Set $B = \bigoplus_{i < \omega}\bbF_p$ and let $w : B \rightarrow \omega \cup \{\infty\}$ be the valuation given by
\[  w((x_i)_{i<\omega}) = \min \{ i : x_i \neq 0 \}. \]
By Proposition 4 of \cite{maalouf} the valued vector space $(B,w)$ is C-minimal and hence dp-minimal (by Theorem A.7 of \cite{simon}).

Th\'eor\`eme 1 of \cite{maalouf} implies that $(A,v)$ and $(B,w)$ are elementarily equivalent. Hence $(A,v)$ is dp-minimal.
\end{proof} 

\begin{remark}
  The last step of the previous proof also follows from results in \cref{sec:uniform_families}.
  Let $(B,w)$ be as in the proof of \cref{prop:valued_vector_space} and set
  \[ B_i = \{ x \in B : w(x) \geq i \}. \]
  Then $A \cong \varprojlim_{i < \omega} B/B_i$ and hence $(A,v)$ is dp-minimal by \cref{limitlemma}.
\end{remark}

\subsection{A quantifier elimination result}\label{sec:qe_result}
We denote the set of primes by $\bbP$. For each prime $p \in \bbP$ we fix an integer $\alpha_p \ge 0$ and a finite $p$-group $A_p$.
Let 
\[ Z \prec \prod_{p \in \bbP} \bbZ_p^{\alpha_p} \times A_p \]
be an abelian group that is (as a pure group) an elementary substructure. We will always assume $Z$ to be infinite.
We fix a set of constants $\{ c_j : j < \omega \} \subseteq Z$ containing $0$ such that the set is dense with respect to the profinite topology on $Z$ and contains every torsion element.
It follows from \cref{p:abelian_structure} that the set of constants is also dense with respect to the profinite topology on  $\prod_{p \in \bbP} \bbZ_p^{\alpha_p} \times A_p$.

\begin{definition}\label{def:decomposition}
If $\pi$ is a set of primes, we set
\[ Z_\pi = Z \cap ( \prod_{p \in \bbP} \bbZ_p^{\alpha_p} \times \prod_{p \in \bbP \setminus \pi} A_p ) \quad \text{and} \quad A_\pi = Z \cap \prod_{p \in \pi}A_p. \]
\end{definition}

Note that we have $Z = Z_\pi \times A_\pi$ for any set $\pi \subseteq \bbP$. The group $A_\pi$ is the $\pi$-torsion part of $Z$ and the group $Z_\pi$ has no $\pi$-torsion.

Let $v$ be a good valuation and
set $I := \omega \cup \{+\infty, -\infty \}$. For each $l \ge 1$ we define a function $v^l : Z \rightarrow I$ by
\[ v^l(a) = \begin{cases} +\infty & \text{iff } a = 0 \\ 
i & \text{iff } a \in lB_i \setminus lB_{i+1} \\
-\infty & \text{iff } a \not \in lZ. \end{cases} \]

Note that if $a \in lZ$, then $v^l(a) = v(a/l)$.
Now $Z$ together with the valuation $v$ may be viewed as a two-sorted structure with group sort $\calZ$ and value sort $\calI$ in the language $\calL^{-} = \calL_\calZ \cup \calL_v \cup \calL_\calI^{-}$ where
\begin{itemize}
	\item $\calL_\calZ = \{+,-, c_j : j < \omega \}$ is the obvious language on $Z$,
	\item $\calL_v = \{ v^l : l \ge 1 \}$ consists of symbols for the functions $v^l$, and 
	\item $\calL_\calI^{-} = \{ \leq, 0, +\infty, -\infty \}$ is the obvious language on $I$. 
\end{itemize}
Since we consider the group $Z$ and the constants $c_j$ to be fixed, this structure only depends on the valuation and we denote it by $(Z,v)$.

We define the following binary relations on $I$:
\begin{itemize}
	\item $\Ind_k^{\pi,l}(i,j) \iff i \leq j$ and $|Z_\pi \cap lB_i : Z_\pi \cap lB_j | \geq k$,
	\item $\Div_{q^k}^{\pi,l}(i,j) \iff i \leq j$ and $q^{k\alpha_q}$ divides $|Z_\pi \cap lB_i : Z_\pi \cap lB_j|$,
\end{itemize}
where $\pi$ is a finite set of primes, $q\in \pi$ is a prime, and $k \geq 0$. 
We set $\Ind_k^{\pi,l}(i,+\infty)$ and $\Div_{q^k}^{\pi,l}(i,+\infty)$ to be always true.

\begin{observation}\label{obs:definability}
\begin{itemize}
\item[(a)]If $q \in \pi$ then $q^{k\alpha_q}$ divides $|Z_\pi \cap lB_i : Z_\pi \cap lB_j|$ if and only if $(Z_\pi \cap lB_i) / (Z_\pi \cap lB_j)$ has an element of order $q^k$. In particular, the predicate $\Div_{q^k}^{\pi,l}$ is definable.

\item[(b)]In the standard model $\Div_{q^k}^{\pi,l}(i,j)$ is equivalent to the statement that $q^k$ divides $| \bbZ_q \cap lB_i : \bbZ_q \cap lB_j|$. In that sense the expression
$| \bbZ_q \cap lB_i : \bbZ_q \cap lB_j|$ makes sense even in non-standard models.

\item[(c)]We have $x \in nZ$ if and only if $v^n(x) \geq 0$. Hence the subgroups $nZ$ are quantifier free $0$-definable.
Since the subgroups $nZ$ generate the profinite topology on $Z$, this implies that the open subgroups $Z_\pi$ are quantifier free 0-definable for finite subsets $\pi \subseteq \bbP$.
Moreover, in that case $A_\pi$ is also quantifier free 0-definable since it is a finite set of constants.
\end{itemize}
\end{observation}

Let $V$ be the set of good valuations on $Z$. We set $T_Z := \bigcap_{v \in V}\Th((Z,v))$ to be the common $\calL^{-}$-Theory of structures $(Z,v)$, $v \in V$. The following quantifier elimination result will be shown in the next sections:
\begin{theorem} \label{qe_theorem}
  Let $\calL_\calI \supseteq \calL_\calI^{-}$ be an expansion on the sort $\calI$ and let $T \supseteq T_Z$ be an expansion of $T_Z$ to the language $\calL = \calL_\calZ \cup \calL_v \cup \calL_\calI$. Suppose that:
  \begin{enumerate}
    \item The relations $\Div_{q^k}^{\pi,l}$ and $\Ind_k^{\pi,l}$ are quantifier free 0-definable modulo $T$.
    \item The successor function on $\calI$ is contained in $\calL_\calI$.
    \item Every $\calL_\calI$-formula is equivalent to a quantifier free $\calL$-formula modulo $T$.
  \end{enumerate}
Then $T$ eliminates quantifiers. 
\end{theorem}

To prove the quantifier elimination result we will need to understand formulas that describe systems of linear congruences.
Therefore we will need to understand linear congruences in models of the theory $T$.

\subsubsection{Linear congruences in $\bbZ$} \label{sec:congruence_bbZ}

We will need generalizations of the following well-known fact:

\begin{fact} \label{f:congruence} A linear congruence $nx \equiv a \mod m$ in $\bbZ$ has a solution if and only if $d = \gcd(n,m)$ divides $a$.
  In that case it has exactly $d$ many solutions modulo $m$. If $s$ is a solution, then a complete system of solutions modulo $m$ is given by
  \[s + tm/d, \quad t = 0, \dots d-1. \]
\end{fact}

\begin{observation}\label{obs:congruence}
\cref{f:congruence} has two important consequences:
\begin{enumerate}
  \item[(a)] If $nx \equiv a \mod m$ has a solution and $n = \gcd(n,m)$, then $n$ divides $a$ and hence $a/n$ is a solution.
  \item[(b)] If $nx \equiv a \mod m$ has a solution, then all solutions agree modulo $m/d$.
\end{enumerate}
\end{observation}
Part (a) will be important since in that case a solution will be determined by the constant $a$. Part (b) tells us that solutions of linear congruences can ``collapse''. We will need to understand this collapsing of solutions.

We now fix a group $\bbA$ of the form $\bbA = \prod_{p \in \bbP} \bbZ_p^{\alpha_p}$, $\alpha_p < \omega$. 
If $n$ is a positive integer, let $n(p)$ be the unique integer such that $n = \prod p^{n(p)}$.
Note that \cref{f:congruence} can be applied to $\bbZ_p$ because $\bbZ_p/k\bbZ_p = \bbZ/p^{k(p)}\bbZ$.

We consider linear congruences
\[ nx \equiv a \mod m \]
in $\bbA$ where $n$ and $m$ are positive integers and $a \in \bbA$. 
Note that solving the above linear congruence is equivalent to solving it in each copy of $\bbZ_p$ in the product $\bbA = \prod_{p \in \bbP}\bbZ_p^{\alpha_p}$:

\begin{lemma} \label{lem:count_solutions}
  \begin{enumerate}
    \item[(a)]  Let $nx \equiv a \mod m$ be a linear congruence in $\bbA$. Write 
      \[ a = (a_{p,i})_{p \in \bbP, i < \alpha_p} \in \bbA =\prod_{p \in \bbP} \bbZ_p^{\alpha_p}. \]
      The solutions for $nx \equiv a \mod m$ in $\bbA$ are exactly the tuples $s = (s_{p,i})_{p \in \bbP, i < \alpha_p}$ 
      where each $s_{p,i}$ is a solution for $nx \equiv a_{p,i} \mod m$ in $\bbZ_p$.
    \item[(b)] Set $d = \gcd(n,m)$. Then the linear congruence $nx \equiv a \mod m$ has a solution if and only if 
      $d$ divides $a$ in $\bbA$ (i.e. $a \in d\bbA$). In that case it has exactly $\prod_{p|d} p^{\alpha_p d(p)}$ many solutions modulo $m$ in $\bbA$.
  \end{enumerate}
\end{lemma}

We call a finite family of linear congruences (and negations of linear congruences) a \emph{system of linear congruences}. Recall B\'ezout's identity: 
\begin{fact}[B\'ezout's identity]
If $a_1, \dots a_n$ are integers, then $\gcd(a_1, \dots a_n)$ is a $\bbZ$-linear combination of $a_1, \dots a_n$.
\end{fact}

We will look at systems of linear congruences where the modulus is fixed:

\begin{proposition} \label{prop:lin_cong_1}
  Let $n_r x \equiv a_r \mod m,\, r \in R,$ be a system $\calS$ of linear congruences in $\bbA$ (where $R$ is a finite set). Set $n = \gcd(n_r : r \in R)$ and $d = \gcd(n,m)$. By B\'ezout's identity we can find integers $z_r$ such that $n = \sum_{r \in R} z_rn_r$. Put $a = \sum_{r \in R}z_ra_r$.
  \begin{enumerate}
  \item[(a)] If the system $\calS$ has a common solution, the solutions of $\calS$ are exactly the solutions of $nx \equiv a \mod m$.
  \item[(b)] Set $k = n/d$ and $d_r = n_r/k$. Then the system $\calS$ has a solution if and only if the system $\calT$:
    \[ d_r x \equiv a_r \mod m, \quad r \in R, \]
    has a solution. Moreover, the systems $\calS$ and $\calT$ have the same number of solutions modulo $m$.
  \end{enumerate}
\end{proposition}
\begin{proof}
  (a) It suffices to show this for each factor in the product $\bbA = \prod_{p \in \bbP}\bbZ_p^{\alpha_p}$. Hence we may assume $\bbA = \bbZ_p$ and $m = p^{m(p)}$. Clearly any common solution of the system $\calS$ solves $nx \equiv a \mod m$.
  
  Now suppose $s$ is a solution of $\calS$ (and hence a solution of $nx \equiv a \mod m$). Then by \cref{f:congruence} all solutions of $nx \equiv a \mod m$ are of the form $s + tm/d$ where $d = \gcd(n,m)$.
  Fix $r \in R$. Now $d$ divides $d_r = \gcd(n_r, m)$, say $d_r = k_rd$.
  Therefore
  \[ s+tm/d = s+tk_rm/d_r \]
  solves $n_rx \equiv a_r \mod m$ for all $t = 0, \dots d-1$ (by \cref{f:congruence}).
  Hence every solution of $nx \equiv a \mod m$ solves $\calS$.
  
  (b) We have $n = \gcd(n_r : r \in R) = \sum_{r \in R} z_rn_r$. 
  If we divide by $k$, we get that $d = \gcd(d_r : r \in R) = \sum_{r \in R} z_rd_r$. 
  We aim to show that $\calS$ has a solution if and only if $\calT$ has a solution.
  If $s$ is a solution for $\calS$, then $ks$ solves $\calT$. Now assume that $\calT$ has a solution. Then by (a) the system $\calT$ has the same solutions as the linear congruence
  \[ dx \equiv a \mod m. \]
  Since we assume that $\calT$ has a solution, this implies that $d = \gcd(d,m)$ divides $a$ (by part (b) of \cref{lem:count_solutions}). Then the linear congruence
  \[ nx \equiv a \mod m \]
  also has solutions by part (b) of \cref{lem:count_solutions} since $d = \gcd(n,m)$ divides $a$. If $s$ solves $nx \equiv a \mod m$, then $ks$ solves $dx \equiv a \mod m$ and hence is a solution for $\calT$.
  This implies that $s$ solves $\calS$. Hence $\calS$ has a solution if and only if $\calT$ has a solution. Moreover, if $\calS$ and $\calT$ have solutions, then by (a) the solutions of $\calS$ are exactly the solutions of $nx \equiv a \mod m$ and the solutions of $\calT$ are exactly the solutions of $dx \equiv a \mod m$. Hence they have the same number of solutions modulo $m$ by part (b) of \cref{lem:count_solutions}.
\end{proof}

We will now consider systems of linear congruences where we vary the modulus:

\begin{lemma}\label{lem:collapsing_Zp}
  Let $nx \equiv a \mod p^m$ be a linear congruence in $\bbZ_p$. Set $d = \gcd(n,p^m)$ and suppose $l > 0$ divides $p^m$ such that $d$ divides $p^m/l$.  
  Then 
  \[ nx \equiv a \mod p^m \quad \text{and} \quad  nx \equiv a \mod dl \]
  have exactly $d$ solutions modulo $p^m$ respectively $dl$ and
  all these solutions agree modulo $l$. 
\end{lemma}
\begin{proof}
The assumption implies that $dl$ divides $p^m$. Therefore
\[ d = \gcd(n,p^m) = \gcd(n,dl). \]
By \cref{f:congruence} the congruences have a solution if and only if $d$ divides $a$. In that case $nx \equiv a \mod p^m$ has exactly $d$ solutions modulo $p^m$ and the congruence $nx \equiv a \mod dl$ has exactly $d$ solutions modulo $dl$. Moreover, by part (b) of \cref{obs:congruence} all these solutions agree modulo $l$.
\end{proof}

\begin{proposition} \label{prop:collapsing}
  Let $nx \equiv a \mod m$ be a linear congruence in $\bbA$. Set $d = \gcd(n,m)$. Suppose $l > 0$ divides $m$ and is such that for all $p | d$ we have $p^{d(p)}|(m/l)$ or $p$ does not divide $(m/l)$. Set 
  \[ k = \prod_{p | d, p | (m/l)} p^{d(p)}. \]
  Then the linear congruences
  \[nx \equiv a \mod m \quad \text{and} \quad nx \equiv a \mod kl \]
  have the same number of solutions modulo $m$ respectively $kl$.
  Moreover, if $X$ is the set of solutions modulo $m$ of $nx \equiv a \mod m$, $Y$ is the set of solutions modulo $kl$ of $nx \equiv a \mod kl$, and $X/l$ and $Y/l$ are the images of $X$ and $Y$ in $\bbA/l\bbA$, then $X/l = Y/l$ and each element in $X/l$ (resp. $Y/l$) has exactly $\prod_{p | d, p | (m/l)} p^{\alpha_pd(p)}$ many preimages in $X$ (resp. $Y$).
\end{proposition}
\begin{proof}
By an application of \cref{lem:count_solutions} it suffices to show this in case $\bbA = \bbZ_p$. Hence we will assume $\bbA = \bbZ_p$.

If $p$ does not divide $d$, then $d$ is a unit in $\bbZ_p$ and hence each of the congruences has a unique solution in $\bbZ_p$.

Hence we may assume $p$ divides $d$. If $p$ does not divide $m/l$, then $m(p) = l(p)$ and $k(p) = 0$. Then $m\bbZ_p = kl\bbZ_p = l\bbZ_p$ and therefore the linear congruences
\[ nx \equiv a \mod m, \quad \quad nx \equiv a \mod kl, \quad \text{and} \quad nx \equiv a \mod l \]
have the same solutions (in $\bbZ_p$). Since solutions modulo $m$ (resp. modulo $kl$) are the same as solutions modulo $l$, each element of $X/l$ (resp. $Y/l$) has a unique preimage in $X$ (resp. $Y$).

Now assume $p$ divides $m/l$. Then by assumption $p^{d(p)}$ divides $m/l$. In that case the result follows by \cref{lem:collapsing_Zp}. Note that each element in $X/l$ (resp. $Y/l$) has exactly $d$ preimages in $X$ (resp. $Y$).
\end{proof}

\subsubsection{Linear congruences in $\calZ$}
Fix $T \supset T_Z$ as in Theorem \ref{qe_theorem} (for a group $Z \prec \prod_{p \in \bbP} \bbZ_p^{\alpha_p} \times A_p$ as in the beginning of \cref{sec:qe_result}). We have $v^{l}( a) \geq i$ if and only if $a \in B_i^{v^l} = lB_i$.
Therefore we will consider certain formulas as linear congruences:
\begin{align*}
	v^{l}( nx-a) \geq i &\iff nx \equiv a \mod{lB_{i}}, \\
  v^{l}( nx-a ) < i &\iff nx \not \equiv a \mod{lB_{i}}. 
\end{align*}
Here $x$ will be a variable and $a$ will be a constant.
 The integer $n$ will be part of the formula. In particular, it will always be a standard integer. 
 Recall that for a subset $\pi \subseteq \bbP$ a natural number $n \geq 1$ is called a $\pi$-number if the prime decomposition of $n$ only contains primes in $\pi$.
 
 We will often work in the $\pi$-torsion free group $Z_\pi$ defined in \cref{def:decomposition}. If we assume that $\pi$ is finite, then by part (c) of \cref{obs:definability} the subgroup $Z_\pi$ is quantifier free 0-definable. If $M \models T$ is any model, then we set $Z_\pi(M)$ to be the
 subgroup defined by the formula which defines $Z_\pi$ in $Z$. The subgroup $A_\pi(M)$ is defined analogously.
 
 If we use the notation in part (b) of \cref{obs:definability}, then
  \[\gcd(n,lB_i) := \gcd(n, \prod_{\{ p : \alpha_p > 0 \}} |\bbZ_p : \bbZ_p \cap lB_i |) \]
  is well-defined even if $lB_i$ has infinite index because $n$ is always a standard integer. 
  Therefore the results in \cref{sec:congruence_bbZ} can be formulated using the divisibility predicates and they will hold true for models of $T$.
  
\begin{proposition} \label{prop:count_solutions_Z}
Let $M$ be a model of $T$ and let $nx \equiv a \mod lB_i$ be a linear congruence in $Z(M)$. Let $\pi$ be a finite set of primes such that $n$ is a $\pi$-number and $a \in Z_\pi(M)$.
Then the linear congruence has a solution in $Z_\pi(M)$ if and only if $d = \gcd(n,lB_i)$ divides $a$ (i.e. $a \in dZ_\pi(M)$). In that case there are exactly $\prod p^{\alpha_p d(p)}$
many solutions modulo $lB_i$ in $Z_\pi(M)$.
\end{proposition}
\begin{proof}
This is essentially part (b) of \cref{lem:count_solutions}. Since this is a first-order statement, it suffices to consider good valuations $v$ on $Z$. Since the statement only affects the quotients $Z/lB_i$, we may assume that $Z$ is of the form
\[ Z = \prod_{p \in \bbP} \bbZ_p^{\alpha_p} \times A_p. \]
Put $\bbA = \prod_{p \in \bbP} \bbZ_p^{\alpha_p}$ and $H = \prod_{p \not \in \pi} A_p$. Then
\[ Z_\pi = \bbA \times H \]
and $H$ has no $\pi$-torsion. Write $a = a_0h$ for $a_0 \in \bbA$ and $h \in H$.
Then we can apply \cref{lem:count_solutions} to the linear congruence
\[ nx \equiv a_0 \mod lB_i \]
in $\bbA$. Note that the linear congruence
\[ nx \equiv h \mod lB_i \]
has a unique solution modulo $lB_i$ in $H$ (namely $h/n$). This shows the proposition.
\end{proof}
  
\begin{proposition} \label{prop:lin_cong_1_Z}
  Let $M$ be a model of $T$ and let $n_r x \equiv a_r \mod lB_i,\, r \in R,$ be a system of linear congruences. Let $\pi$ be a finite set of primes such that all $n_r$ are $\pi$-numbers and all $a_r$ are contained in $Z_\pi(M)$. Set $n = \gcd(n_r : r \in R)$ and $d = \gcd(n,lB_i)$. By B\'ezout's identity we can find integers $z_r$ such that $n = \sum_{r \in R} z_rn_r$. Put $a = \sum_{r \in R}z_ra_r$.
  \begin{enumerate}
  \item[(a)] If the system has a common solution in $Z_\pi (M)$, the solutions of the system in $Z_\pi (M)$ are exactly the solutions of $nx \equiv a \mod lB_i$ in $Z_\pi (M)$.
  \item[(b)] Set $k = n/d$ and $d_r = n_r/k$. Then the system has a solution in $Z_\pi(M)$ if and only if the system
    \[ d_r x \equiv a_r \mod lB_i, \quad r \in R, \]
    has a solution in $Z_\pi(M)$. Moreover, these systems have the same number of solutions modulo $lB_i$ in $Z_\pi(M)$.
  \end{enumerate}
\end{proposition}
\begin{proof}
This follows from \cref{prop:lin_cong_1} by the same arguments that are used in \cref{prop:count_solutions_Z}.
\end{proof}
  
\begin{proposition} \label{prop:collapsing_Z}
  Let $M$ be a model of $T$ and let $nx \equiv a \mod lB_i$ be a linear congruence. 
  Let $\pi$ be a finite set of primes such that $n$ is a $\pi$-number and $a \in Z_\pi(M)$.
  Set $d = \gcd(n,lB_i)$. Fix $u > 0$ and $j > i$ such that $ulB_j < lB_i$ is a subgroup and is such that for all $p | d$ we have that $p^{d(p)}$ divides $|\bbZ_p \cap lB_i : \bbZ_p \cap ulB_i|$ or $p$ does not divide $|\bbZ_p \cap lB_i : \bbZ_p \cap ulB_i|$. Set 
  \[ k = \prod\{ p^{d(p)} \,: \,  p \text{ divides } d \text{ and } |\bbZ_p \cap lB_i : \bbZ_p \cap ulB_i| \}. \]
  Then the linear congruences
  \[nx \equiv a \mod lB_i \quad \text{and} \quad nx \equiv a \mod kulB_j \]
  have the same number of solutions modulo $lB_i$ respectively $kulB_j$ in $Z_\pi(M)$.
  Moreover, if $X$ is the set of solutions modulo $lB_i$ of $nx \equiv a \mod lB_i$, $Y$ is the set of solutions modulo $kulB_j$ of $nx \equiv a \mod kulB_j$, and $X/ulB_j$ and $Y/ulB_j$ are the images of $X$ and $Y$ modulo $ulB_j$, then $X/ulB_j = Y/ulB_j$ and each element in $X/ulB_j$ (resp. $Y/ulB_j$) has exactly $\prod_{p | d, p | (m/l)} p^{\alpha_pd(p)}$ many preimages in $X$ (resp. $Y$).
\end{proposition}
\begin{proof}
This follows from \cref{prop:collapsing} by the same arguments that are used in \cref{prop:count_solutions_Z}.
\end{proof}
  
  The following lemma will often be useful:
 \begin{lemma} \label{lem:multiply_congruence}
 Fix a model $M \models T$, let $\pi$ be a finite set of primes, let $a \in Z_\pi(M)$, and let $t$ be a $\pi$-number. Then the linear congruences 
 \[ nx \equiv a \mod{lB_{i}} \quad \text{and} \quad tnx \equiv ta \mod{tlB_{i}} \]
 have the same solutions in $Z_\pi(M)$.
 \end{lemma}
 \begin{proof}
  Multiplying by $t$ is injective since $Z_\pi(M)$ does not have $t$-torsion.
 \end{proof}
   
\subsubsection{Systems of linear congruences in $\calZ$}
Fix $T$ as in Theorem \ref{qe_theorem}. Note that we assume that the successor function $S$ (on $\calI$) is contained in $\calL_\calI$.

\begin{lemma} \label{lem:solution_quotient}
	Let $M_1$ and $M_2$ be models of $T$ and let $(A,J)$ be a common substructure. Let $\pi$ be a finite set of primes and let $\calS:$
	\[ n_r x \equiv a_r \mod lB_{i}, \quad r \in R, \]
	be a system of linear congruences where each $n_r$ is a $\pi$-number, $a_r \in Z_\pi(A)$, and $i \in J$.
	Suppose there is a $\pi$-number $t$ and a constant $b \in A$ such that $t$ divides $b$ and $b/t$ solves $\calS$ in $Z_\pi(M_1)$.
	Then $b/t$ solves $\calS$ in $Z_\pi(M_2)$.
\end{lemma}
\begin{proof}
  We have $t$ divides $b$ if and only if $v^t(b) \geq 0$. This does not depend on the model. Moreover, $b/t$ solves $n_r x \equiv a_r \mod lB_{i}$ if and only if
  $v^l(n_r (b/t) -a_r ) \geq i$. By \cref{lem:multiply_congruence} we have
  \[ v^l(n_r (b/t) -a_r ) = v^{tl}(n_rb - ta_r ). \]
 Therefore this value does not depend on the model. 
\end{proof}

\begin{lemma} \label{lem:lin_cong_Z_1}
	Let $M_1$ and $M_2$ be models of $T$ and let $(A,J)$ be a common substructure. Let $\pi$ be a finite set of primes and let
	\[ n_r x \equiv a_r \mod lB_{i}, \quad r \in R, \]
	be a system of linear congruences where each $n_r$ is a $\pi$-number, $a_r \in Z_\pi(A)$, and $i \in J$.
	Then the system has the same number of solutions modulo $lB_i$ in $Z_\pi(M_1)$ and $Z_\pi(M_2)$. 
\end{lemma}
\begin{proof}
	Set $n := \gcd(n_r : r \in R)$, say $n = \sum_{r \in R}z_rn_r$ (by B\'ezout's identity), and put $a := \sum_{r \in R} z_ra_r$. Set $d := \gcd(n,\prod_{\{p: \alpha_p > 0\}}|\bbZ_p : \bbZ_p \cap lB_{i}|)$, $k = n/d$, and $d_r = n_r/k$. Then by \cref{prop:lin_cong_1_Z} (b) the system 
	\[n_r x \equiv a_r \mod lB_{i},\, r \in R,\]
	has the same number of solutions modulo $lB_{i}$ in $Z_\pi(M_1)$ (resp. $Z_\pi(M_2)$) as the system 
	\[d_rx \equiv a_r \mod lB_{i}, \, r \in R.\]
	
	We have $d = \gcd(d_r : r \in R) = \sum_{r \in R}z_rd_r.$ By \cref{prop:lin_cong_1_Z} (a) any solution of the system 
	\[ d_r x \equiv a_r \mod lB_{i}, \, r \in R, \]
	is a solution of $d x \equiv a \mod lB_{i}$. Now by \cref{prop:count_solutions_Z} the linear congruence $dx \equiv a \mod lB_i$ has a solution if and only if $d$ divides $a$. In that case $a/d$ must be a solution and we can apply \cref{lem:solution_quotient} to see that this must hold true in both models.
	
	Hence $Z_\pi(M_1)$ contains a solution if and only if $Z_\pi(M_2)$ contains a solution. In that case the solutions are exactly the solutions of $nx \equiv a \mod lB_i$ and by \cref{prop:count_solutions_Z} the number of solutions modulo $lB_i$ does not depend on the model.
\end{proof}

\begin{lemma} \label{lem:lin_cong_Z_2}
	Let $M_1$ and $M_2$ be models of $T$ and let $(A,J)$ be a common substructure. Let $\pi$ be a finite set of primes and let $\mathcal{S}$ be a system
	\begin{align*}
	n_r x \equiv a_r \mod lB_{i_r}, \quad & r \in R,
	\end{align*}
	of linear congruences where $l$ and each $n_r$ is a $\pi$-number, $a_r \in Z_\pi(A)$, $i_r \in J$. Suppose moreover, that the index  $|B_{i_r} : B_{i_{r'}}|$ is a $\pi$-number whenever it is finite.  Fix $r_\text{max} \in R$ such that $i_{r_\text{max}}$ is maximal. Then $\mathcal{S}$ has the same number of solutions modulo $lB_{i_{r_\text{max}}}$ in $Z_\pi(M_1)$ and $Z_\pi(M_2)$.
\end{lemma}
\begin{proof}
  If $|lB_{i_r}:lB_{i_{r'}}|$ is finite, then there is a $\pi$-number $t$ such that $lB_{i_r} = tlB_{i_{r'}}$.
  \cref{lem:multiply_congruence} allows us to replace the linear congruence
  \[ n_{r'} x \equiv a_{r'} \mod lB_{i_{r'}} \]
  by the linear congruence
  \[tn_{r'} x \equiv ta_{r'} \mod lB_{i_r}. \]
  Hence we may assume that the index $|lB_{i_r}:lB_{i_{r'}}|$ is infinite whenever $i_r < i_{r'}$.
  
  For $r_0 \in R$ set $R[r_0] = \{ r \in R : i_r = i_{r_0} \}$ and consider the system $\calS_{r_0}$:
  \[ n_r x \equiv a_r \mod lB_{i_r}, \quad r \in R[r_0]. \]
	By \cref{lem:lin_cong_Z_1} the system $\calS_{r_0}$ has the same number of solutions modulo $lB_{i_{r_0}}$ in $Z_\pi(M_1)$ and $Z_\pi(M_2)$.
	If $\calS_{r_0}$ has no solution, then $\calS$ has no solution and we are done. Hence assume that $\calS_{r_0}$ has a solution (and hence has the same number of solutions in both models by \cref{lem:lin_cong_Z_1}).
	
	Then by \cref{prop:lin_cong_1_Z} we can replace the system $\calS_{r_0}$ by a single linear congruence without changing the solutions.
	
	Hence we may assume 
	\[i_r = i_{r'} \iff r = r'\]
	for all $r,r' \in R$.	
	Now we may write $R = \{ r_0, \dots r_m \}$ such that $i_{r_0} > \dots > i_{r_m}$. We prove the lemma by induction on $m$.
	The case $m=0$ is done by \cref{lem:lin_cong_Z_1}. Hence we assume $m>0$.
	
	The system $\mathcal{S}$ has the form
	\begin{align*}
		n_{r_0} x &\equiv a_{r_0} \mod lB_{i_{r_0}}, \\
		& \vdots \\
		n_{r_m} x &\equiv a_{r_0} \mod lB_{i_{r_m}}.
	\end{align*}
	
	Now set $d := \gcd( n_{r_0}, \prod_{p \in \bbP, \alpha_p>0} |\bbZ_p : \bbZ_p \cap lB_{i_{r_0}}|)$ and put
	\[ u = \prod_{p|d, \alpha_p>0} \{ |\bbZ_p \cap lB_{i_{r_1}} : \bbZ_p \cap lB_{i_{r_0}}| : |\bbZ_p \cap lB_{i_{r_1}}: \bbZ_p \cap lB_{i_{r_0}}| \text{ is finite} \}. \]
	Set $k = \prod \{ p^{d(p)} : \alpha_p > 0 \text{ and } $p$ \text{ does not divide } u \}$ and consider the system $\mathcal{S}':$
	
	\begin{align*}
		n_{r_0} x &\equiv a_{r_0} \mod kulB_{i_{r_1}}, \\
		un_{r_1} x &\equiv ua_{r_1} \mod ulB_{i_{r_1}}, \\
		un_{r_2} x &\equiv ua_{r_2} \mod ulB_{i_{r_2}}, \\
		& \vdots \\
		un_{r_m} x &\equiv ua_{r_m} \mod ulB_{i_{r_m}}.
	\end{align*}
	
	By \cref{prop:collapsing_Z} the linear congruences $n_{r_0} x \equiv a_{r_0} \mod lB_{i_{r_0}}$ and 
	$n_{r_0} x \equiv a_{r_0} \mod kulB_{i_{r_1}}$ have the same number of solutions modulo $lB_{i_{r_0}}$ respectively $kulB_{i_{r_1}}$
	and the sets of solutions agree modulo $ulB_{i_{r_1}}$. The statement about the number of preimages in \cref{prop:collapsing_Z} implies that $\calS$ and $\calS'$ have the same number of solutions modulo 
	 $lB_{i_{r_0}}$ respectively $kulB_{i_{r_1}}$. By \cref{lem:multiply_congruence} we can rewrite $\calS'$ as follows:
	 
	 \begin{align*}
		n_{r_0} x &\equiv a_{r_0} \mod kulB_{i_{r_1}}, \\
		kun_{r_1} x &\equiv kua_{r_1} \mod kulB_{i_{r_1}}, \\
		kun_{r_2} x &\equiv kua_{r_2} \mod kulB_{i_{r_2}}, \\
		& \vdots \\
		kun_{r_m} x &\equiv kua_{r_m} \mod kulB_{i_{r_m}}.
	\end{align*}
	
	By induction the system $\calS'$ has the same number of solutions modulo $kulB_{i_{r_1}}$ in $Z_\pi(M_1)$ and $Z_\pi(M_2)$.
	Hence $\calS$ has the same number of solutions modulo $lB_{i_{r_0}}$ in $Z_\pi(M_1)$ and $Z_\pi(M_2)$.
\end{proof}

To deal with the general case
we will make use of the following:
\begin{fact}[Inclusion-exclusion priciple]\label{fact:inclusion-exclusion}
Let $A_1, \dots A_n$ be finite sets. Then
\[ |\bigcup_{i = 1}^nA_i| = \sum_{\emptyset \neq J \subseteq \{1, \dots n\}} (-1)^{|J|+1} |\bigcap_{j\in J} A_j|. \]
\end{fact}

\begin{proposition} \label{prop:back-and-forth}
	Let $M_1$ and $M_2$ be models of $T$ and let $(A,J)$ be a common substructure. Let $\pi$ be a finite set of primes and let $\mathcal{S}$ be a system
	\begin{align*}
	n_r x \equiv a_r \mod lB_{i_r}, \quad & r \in R, \\
	n_s x \not\equiv a_s \mod lB_{i_s}, \quad & s \in S,
	\end{align*}
	of linear congruences where each $n_t$ is a $\pi$-number, $a_t \in Z_\pi(A)$, $i_t \in J$ for all $t \in R \cup S$. Assume there is $r_0 \in R$ such that $i_{r_0}$ is maximal in $\{i_t : t \in R \cup S\}$. Suppose moreover, that the index  $|B_{i_t} : B_{i_{t'}}|$ is a $\pi$-number whenever it is finite. Then $\mathcal{S}$ has the same number of solutions modulo $lB_{i_{r_0}}$ in $Z_\pi(M_1)$ and $Z_\pi(M_2)$.
\end{proposition}
\begin{proof}
    For pairwise distinct $s_1, \dots s_n\in S$ let $A_{s_1, \dots s_n}$ be the set of solutions modulo $lB_{i_{r_0}}$ of the system $\calS_{s_1, \dots s_n}$:
    \begin{align*}
      n_r x &\equiv a_r \mod lB_{i_r}, \quad r \in R, \\
      n_{s_1} x &\equiv a_{s_1} \mod lB_{i_{s_1}}, \\
      \vdots \\
      n_{s_n} x &\equiv a_{s_n} \mod lB_{i_{s_n}}.
    \end{align*}
    
    By \cref{lem:lin_cong_Z_2} the system $\calS_{s_1, \dots s_n}$ has the same number of solutions in $Z_\pi(M_1)$ and $Z_\pi(M_2)$. In particular, this holds true for the system $\calS_\emptyset$:
    \[ n_r x \equiv a_r \mod lB_{i_r}, \quad r \in R. \]
    Moreover, $\bigcup_{s \in S} A_s$ is exactly the set of solutions modulo $lB_{i_{r_0}}$ for $\calS_\emptyset$ that do not solve $\calS$.
    
    Note that $A_{s_1} \cap \dots \cap A_{s_n} = A_{s_1, \dots s_n}$ and hence by an application of the inclusion-exclusion principle the number $| \bigcup_{s \in S}A_s|$ (which is finite since we only count solutions modulo $lB_{i_{r_0}}$) does not depend on the model.
    
    Now the system $\calS$ is solved by exactly
    \[ |A_\emptyset| -  |\bigcup_{s \in S} A_s| \]
    many solutions modulo $lB_{i_{r_0}}$ and this number does not depend on the model.
\end{proof}

\subsubsection{Proof of quantifier elimination} \label{subsec:proof_qe}

\begin{proof}[Proof of Theorem \ref{qe_theorem}]
By \cref{lem:qe} it suffices to show that every formula of the form
\[ \psi(\bar{z}, \bar{i}) \equiv \exists x \in \calZ \bigwedge_{r \in R} \phi_r(x,\bar{z},\bar{i}) \]
is equivalent to a quantifier free formula modulo $T$
where $\bar{z} \subseteq \calZ$, $\bar{i} \subseteq \calI$ and each $\phi_r$ is either a basic $\calL_Z$-formula or is of the form
\[ v^{l_r}(t_r(x,\bar{z})) = i_r \]
where $t_r$ is an $\calL_Z$-term and $i_r$ is one variable in the tuple $\bar{i}$.

Write $R = R_0 \cup R_1 \cup R_2$ such that
\begin{align*}
	\phi_r(x,\bar{z},\bar{i}) &\equiv n_rx-t_r(\bar{z}) = 0, \quad \text{for } r \in R_0, \\
	\phi_r(x,\bar{z},\bar{i}) &\equiv n_rx-t_r(\bar{z}) \neq 0, \quad \text{for } r \in R_1, \quad \text{and} \\
	\phi_r(x,\bar{z},\bar{i}) &\equiv v^{l_r}(n_rx - t_r(\bar{z})) = i_r, \quad \text{for } r \in R_2.
\end{align*}

Now let $\pi$ be a finite set of primes such that $n_r$, $l_r$, and the cardinalities of all finite quotients $|lB_{i_r}:lB_{i_{r'}}|$ are $\pi$-numbers.
Fix two models $M_1,M_2$ of $T$ and let $(A,J)$ be a common substructure, $\bar{a} \subseteq A, \bar{\eta} \subseteq J$.
Set $a_r := t_r(\bar{a})$. We have $A = Z_\pi(A) \times A_\pi(A)$ (since $A_\pi$ is a finite set of constants) and hence each $a_r$ can be written as $a_r = a_r^\pi b_r^\pi$ with $a_r^\pi \in Z_\pi(A)$ and $b_r^\pi \in A_\pi(A)$. Now suppose the formula $\psi(\bar{a},\bar{\eta})$ has a solution in $M_1$. By \cref{prop:qe_criterium} it suffices to show that it has a solution in $M_2$.

If $r \in R_0$, then $\phi_r$ must be satisfied in $Z_\pi(M_1)$ and $A_\pi(M_1)$. If $r \in R_1$, then it suffices if $\phi_r$ is satisfied in $Z_\pi(M_1)$ or $A_\pi(M_1)$.
If $r \in R_2$, then we have
\[ \phi_r(x, \bar{a},\bar{\eta}) \equiv v^{l_r}(n_r x - a_r ) = i_r. \]
This is satisfied if we have ``='' in $Z_\pi(M_1)$ or $A_\pi(M_1)$ and ``$\geq$'' in the other subgroup.
Hence there are subsets $R_1^\pi \subseteq R_1$ and $R_2^\pi \subseteq R_2$ such that the formulas

\begin{align*}
	\psi^\pi \equiv \exists x \in Z_\pi &\bigwedge_{r \in R_0} n_rx-a_r^\pi = 0 \\
		&\land \bigwedge_{r \in R_1^\pi} n_rx-a_r^\pi \neq 0 \\
		&\land \bigwedge_{r \in R_2^\pi} v^{l_r}(n_rx-a_r^\pi) = i_r \\
		&\land \bigwedge_{r \in R_2 \setminus R_2^\pi}  v^{l_r}(n_rx-a_r^\pi) \geq i_r
\end{align*}
and
\begin{align*}
	\overline{\psi^\pi} \equiv \exists x \in A_\pi &\bigwedge_{r \in R_0} n_rx-b_r^\pi = 0 \\
		&\land \bigwedge_{r \in R_1 \setminus R_1^\pi} n_rx-b_r^\pi \neq 0 \\
		&\land \bigwedge_{r \in R_2 \setminus R_2^\pi} v^{l_r}(n_rx-b_r^\pi) = i_r \\
		&\land \bigwedge_{r \in R_2^\pi}  v^{l_r}(n_rx-b_r^\pi) \geq i_r 
\end{align*}
have a solution in $Z_\pi(M_1)$ respectively $A_\pi(M_1)$. Since $A_\pi$ is a finite set of constants, this implies that $\overline{\psi^\pi}$ has a solution in $A_\pi(M_2)$. It remains to show that $\psi^\pi$ has a solution in $Z_\pi(M_2)$.

If $R_2 = \emptyset$, then we are done since the formulas $x \in nZ$ are quantifier free $0$-definable and hence the result follows from the usual quantifier elimination for abelian groups. Therefore we assume $R_2 \neq \emptyset$.

If $R_0 \neq \emptyset$, say $r_0 \in R_0$, then $a_{r_0}^\pi/n_{r_0}$ is the solution of $\psi^\pi$ in $Z_\pi(M_1)$.
\cref{lem:solution_quotient} implies that $a_{r_0}^\pi/n_{r_0}$ also solves $\psi^\pi$ in $Z_\pi(M_2)$. Hence we may assume $R_0 = \emptyset$.

If $i_r = +\infty$ for some $r$, then we have
\[ v^{l_r}(n_rx - a_r^\pi ) \geq i_r \iff n_rx - a_r^\pi = 0. \]
Hence we may assume $i_r < +\infty$ for all $r \in R_2$.

Given $l' \geq 1$ there is a finite set of constants $C_{l'}$ in the language such that the formula $v^{l'}(t(x,\bar{z})) = -\infty$ is equivalent to
\[ \bigvee_{c \in C_{l'}} v^{l'}(t(x,\bar{z})-c) \geq 0.  \]
Thus we may also assume $i_r > -\infty$ for all $r \in R_2$.

Note that each formula of the form $n_rx-a_r^\pi \neq 0$ excludes only a single solution. Since we assume $R_2 \neq \emptyset$ and all formulas of the form
\[ v^{l_r}(n_rx - a_r^\pi ) = i_r \text{ or } v^{l_r}(n_rx-a_r^\pi ) \geq i_r \]
are solved by cosets of $l_rB_{i_r+1}$, we may moreover assume $R_1 = \emptyset$.

By \cref{lem:multiply_congruence} we have $v^{l_r}(n_rx-a_r^\pi) = v^{ml_r}(mn_rx-ma_r^\pi)$ for all $\pi$-numbers $m$. Thus we may use \cref{lem:multiply_congruence} to replace each $l_{r'}$ by
$l := \lcm(l_r : r \in R_2)$.  

We consider formulas as linear congruences:
\begin{align*}
	v^l(n_rx-a_r^\pi) = i_r &\iff (n_rx - a_r^\pi \equiv 0 \mod lB_{i_r} \quad \land \quad n_rx-a_r^\pi \not \equiv 0 \mod lB_{i_r+1}), \\
	v^l(n_rx-a_r^\pi) \geq i_r &\iff n_rx-a_r^\pi \equiv 0 \mod lB_{i_r}.
\end{align*}
Hence it suffices to show that the system of linear congruences
\begin{align*}
	n_rx-a_r^\pi &\equiv 0 \mod lB_{i_r}, \quad r \in R_2, \\
	n_rx-a_r^\pi &\not \equiv 0 \mod lB_{i_r+1}, \quad r \in R_2^\pi,
\end{align*}
has a solution in $Z_\pi(M_2)$. After slightly adjusting the system (by using \cref{lem:multiply_congruence}) and renaming, we get a system
\begin{align*}
	n_sx-b_s &\equiv 0 \mod lB_{i_s}, \quad s \in S, \\
	n_tx-b_t &\not \equiv 0 \mod lB_{i_t}, \quad t \in T,
\end{align*}
where $S \neq \emptyset$ and every index $|B_{i_r}:B_{i_{r'}}|, r,r' \in S\cup T$ is infinite or trivial.
If there is an element $s \in S$ such that $i_s$ is maximal in $\{i_r : r \in S \cup T\}$, then we are done by \cref{prop:back-and-forth}.
Hence suppose there is $t_0 \in T$ such that $i_{t_0} > i_s$ for all $s \in S$. Then $|B_{i_s}:B_{i_{t_0}}|$ is infinite for all $s \in S$. In particular, the congruence
\[n_{t_0}x-b_{t_0} \not \equiv 0 \mod lB_{i_{t_0}} \]
can be ignored, since each $lB_{i_s}$-class consists of infinitely many $lB_{t_0}$ classes. Hence we removed one linear congruence from the system. After iterating this, we can find $s \in S$ such that $i_s$ is maximal.
\end{proof}

\subsection{The monotone hull} \label{sec:monotone_hull}
Theorem \ref{qe_theorem} gives quantifier elimination up to a suitable language on $\calI$.
The following gives a tame expansion of $\calL_\calI^{-}$ which allows us to analyze the definable sets.

A binary relation $R$ on a linear ordering is called \textit{monotone} if and only if it satisfies
\[x' \leq x R y \leq y' \quad \text{implies} \quad x'Ry'.\]
The following result by Simon states that expanding a linear ordering by monotone binary relations is tame: 

\begin{proposition}[Proposition 4.1 and Proposition 4.2 of \cite{dp-minimal_ordered}] \label{dp-min_order}
	Let $(I, \leq, R_\alpha, C_\beta)_{\alpha,\beta}$ be a linear order equipped with monotone binary relations and unary predicates such that every $\emptyset$-definable monotone binary relation is given by one of the $R_\alpha$ and every $\emptyset$-definable unary predicate is given by one of the $C_\beta$. Then $(I,\leq, R_\alpha,C_\beta)_{\alpha, \beta}$ has quantifier elimination and is dp-minimal.
\end{proposition}

Fix a theory $T_Z$ as in the quantifier elimination statement and let $M \models T_Z$ be a model.
Note that the definable relations $\leq$, $\Div_{q^k}^{\pi,l}$, and $\Ind_k^{\pi,l}$ are monotone.

\begin{definition} \label{def:L_mon}
Let $S$ be a set of unary predicates and monotone binary relations on the value set of $M$.
\begin{enumerate}
\item[(a)] We define $\calL_{\calI,\text{mon}}^S$ to be the monotone hull of
\[ \calL_\calI^S := \calL_\calI^- \cup \{ \Div_{q^k}^{\pi,l}, \Ind_k^{\pi,l} \}_{q,\pi,l,k} \cup S, \]
i.e. the expansion of $\calL_\calI^S$ by all 0-definable (in $\calL_\calI^S$) unary relations and all 0-definable monotone binary relations on the value sort.
\item[(b)] Set $\calL_\mathrm{mon}^S = \calL_\calZ \cup \calL_v \cup \calL_{\calI,\mathrm{mon}}^S$ and define $\calL_\mathrm{mon} = \calL_\mathrm{mon}^\emptyset$. 
\end{enumerate}
\end{definition}
Note that $\calL_\mathrm{mon} \supseteq \calL^-$ is an expansion by definitions.

\begin{proposition} \label{qe_monotone} Let $S$ be as in \cref{def:L_mon}.
 Then $\Th(M)$ admits quantifier elimination in the language $\calL_\mathrm{mon}^S$.
\end{proposition}
\begin{proof}
The successor function and its inverse are $0$-definable. If $R \in \calL_{\calI,\text{mon}}^S$ is a monotone binary relation, then so is $R_{m,n}(x,y) \iff R(x+m,y+n)$ for all $m,n \in \bbZ$. The same holds true for $0$-definable unary predicates. Therefore adding the successor function to the language does not add any new definable sets in $\calI$.
Hence \cref{qe_theorem} and \cref{dp-min_order} imply quantifier elimination in $\calL_\mathrm{mon}^S$.
\end{proof}

\subsection{Dp-minimality and distality} \label{sec:dp-min_distal}
Let $T$ be a complete $\calL_\mathrm{mon}^S$-theory as in \cref{qe_monotone}.

\begin{lemma}\label{lem:dp-min_value}
  Let $(Z,I,v) \models T$ be a sufficiently saturated model and let $(a_j)_{j\in J_1}$ and $(b_j)_{j\in J_2}$ be mutually indiscernible sequences in the group sort. Let $\gamma \in I$ be a singleton.
  Then one of the sequences is indiscernible over $\gamma$.
\end{lemma}
\begin{proof}
  We may assume that both sequences are indexed by a dense linear order.
  Suppose $(a_j)_{j\in J_1}$ is not indiscernible over $\gamma$. By the quantifier elimination result this must be witnessed by a formula of the form
  \[R(v^l(t(\bar{x})), \gamma) \]
   where $t$ is an $\calL_\calZ$-term, $R$ is a monotone binary relation on $I$, and $l \geq 1$.
  Hence we can find tuples $\bar{j_0},\bar{j_1} \subseteq J_1$ of the same order type such that
  \[ \models R(v^l(t(\overline{a}_{\overline{j_0}})),c) \quad \text{and} \quad \not \models R(v^l(t(\overline{a}_{\overline{j_1}})),c)  \]
  where $\overline{a}_{\overline{j_i}} = (a_j)_{j\in\overline{j_i}}$ is the tuple corresponding to $\overline{j_i} \subseteq J_1$.
  
  After replacing $\bar{j_0}$ or $\bar{j_1}$ if necessary, we may assume that $\bar{j_0}$ and $\bar{j_1}$ have disjoint convex hulls in $J_1$.
  We can extend to a sequence $(\bar{j_i})_{i<\omega}$ such that $(\overline{a}_{\overline{j_i}})_{i<\omega}$ is an indiscernable sequence.
  Then 
  \[ (v^l(t(\overline{a}_{\overline{j_i}})))_{i<\omega} \]
  is a non-constant indiscernible sequence in the value sort that is not indiscernible over $\gamma$.
  
  By \cref{dp-min_order} the value sort is dp-minimal. Therefore $(b_j)_{j\in J_2}$ must be indiscernible over $\gamma$: Otherwise we could apply the above argument to
  the sequence $(b_j)_{j\in J_2}$ to get a second non-constant indiscernible sequence in the value sort which is not indiscernible over $\gamma$. Since these two sequences would be mutually indiscernible, this would contradict dp-minimality of the value sort.
\end{proof}

\begin{theorem} \label{thm:dp-minimality} 
  $T$ is dp-minimal.
\end{theorem}
\begin{proof}
  Let $M = (Z,I,v) \models T$ be a sufficiently saturated model and let $J_1$ and $J_2$ be mutually indiscernible sequences. We will assume that both of them are indexed by a dense linear order. Let $z \in Z$ be a singleton.
  We aim to show that one of the sequences is indiscernible over $z$.
  
  Since $I$ is essentially an imaginary sort, we may assume that the sequences $J_1$ and $J_2$ live in the sort $Z$. Note that equality on the value sort can be expressed using the monotone binary relation $\leq$. By the quantifier elimination result, the failure of indiscernibility must be witnessed by formulas of the following form:
  \begin{enumerate}
    \item $t(\bar{x})-nz = 0$,
    \item $C(v^l(t(\bar{x})-nz))$,
    \item $R(v^{l_1}(t_1(\bar{x})-n_1z),v^{l_2}(t_2(\bar{x})-n_2z))$,
  \end{enumerate}  
  where $t$ is an $\calL_\calZ$-term, $C$ is a coloring on $\calI$, $R$ is a monotone binary relation on $\calI$, $l \geq 1$, and $n \in \bbZ$. One of the terms in the third case could also be a quantifier free 0-definable constant in the value sort. This case is analogous to case (b) below and therefore we will not consider it separately.
  
  Note that a formula of the first type would imply that $z$ is algebraic over the parameters plugged in for $\bar{x}$. Hence it suffices to consider the other two types of formulas.
  If an indiscernible sequence $J$ is not indiscernible over $z$, then this must be witnessed by $\bar{a},\bar{a}' \subseteq J$ of the same order type such that we are in one of the following cases:
  \begin{enumerate}
  \item[(a)] We have
  \[ v^{l_1}(t(\bar{a})-nz ) \neq v^{l_1}(t(\bar{a}')-nz ) \quad \text{and} \quad v^{l_2}(t'(\bar{a})-n'z ) \neq v^{l_2}(t'(\bar{a}') - n'z ) \]
  and
  \[ \models R(v^{l_1}(t(\bar{a})-nz ),v^{l_2}(t'(\bar{a})-n'z )) \quad \text{and} \quad \not \models R(v^{l_1}(t(\bar{a}')-nz ),v^{l_2}(t'(\bar{a}')-n'z )) \]
  for some choice of $t,t', n \neq 0, n' \neq 0,$ and a relation $R$.
  
  \item[(b)] We have
  \[ v^{l_1}(t(\bar{a})-nz ) \neq v^{l_1}(t(\bar{a}')-nz ) \]
  and
  \[ \models R(v^{l_1}(t(\bar{a})-nz ),v^{l_2}(t'(\bar{a}) )) \quad \text{and} \quad \not \models R(v^{l_1}(t(\bar{a}')-nz ),v^{l_2}(t'(\bar{a}') )) \]
  for some choice of $t,t', n \neq 0,$ and a relation $R$.
  
  \item[(c)] We have 
  \[ v^{l_1}(t(\bar{a})-nz ) = v^{l_1}(t(\bar{a}')-nz )  \quad \text{and} \quad v^{l_2}(t'(\bar{a})) < v^{l_2}(t'(\bar{a}'))\]
  and
  \[ \not \models R(v^{l_1}(t(\bar{a})-nz ),v^{l_2}(t'(\bar{a}) )) \quad \text{and} \quad \models R(v^{l_1}(t(\bar{a}')-nz ),v^{l_2}(t'(\bar{a}') )) \]
  or
  \[ \models R(v^{l_2}(t'(\bar{a}), v^{l_1}(t(\bar{a})-nz ) )) \quad \text{and} \quad \not \models R( v^{l_2}(t'(\bar{a}') ), v^{l_1}(t(\bar{a}')-nz )) \]
  for some choice of $t,t', n \neq 0,$ and a monotone binary relation $R$.
  
  \item[(d)] We have
  \[ v^{l_1}(t(\bar{a})-nz ) = v^{l_1}(t(\bar{a}')-nz )  \quad \text{and} \quad v^{l_2}(t'(\bar{a}) -n'z ) < v^{l_2}(t'(\bar{a}') -n'z )\]
  and
  \[ \not \models R(v^{l_1}(t(\bar{a})-nz ),v^{l_2}(t'(\bar{a})-n'z )) \quad \text{and} \quad \models R(v^{l_1}(t(\bar{a}')-nz ),v^{l_2}(t'(\bar{a}')-n'z )) \]
  or
  \[ \models R(v^{l_2}(t'(\bar{a})-n'z), v^{l_1}(t(\bar{a})-nz ) )) \quad \text{and} \quad \not \models R( v^{l_2}(t'(\bar{a}')-n'z ), v^{l_1}(t(\bar{a}')-nz )) \]
  for some choice of $t,t', n \neq 0, n' \neq 0,$ and a monotone binary relation $R$.
  \end{enumerate}
  The case corresponding to a coloring is essentially the same as (b) so we will not do it explicitly.
  
  We will use \cref{lem:multiply_congruence} to assume that all the $l_i$ coincide:
  Let $\pi$ be a finite set of primes. We want to be able to work in $Z_\pi(M)$. Fix a term
  \[v^l(t(\bar{a})-nz) \]
  and write $t(\bar{a}) = b_0(\bar{a})+b_1(\bar{a})$, $z = c_0+c_1$ for $b_0(\bar{a}),c_0 \in Z_\pi(M)$, $b_1(\bar{a}),c_1 \in A_\pi(M)$.
  Since $A_\pi(M)$ is a finite set of constants, the value of $b_1(\bar{a})$ only depends on the order type of $\bar{a}$. Therefore
  \[ \gamma = v^l(b_1(\bar{a})+nc_1 ) \in A_\pi(M) \]
  also only depends on the order type of $\bar{a}$. We have
  \[ v^l(t(\bar{x}) - nz ) = \min \{ v^l(b_0(\bar{x})-nc_0), \gamma \} \]
  because $Z = Z_\pi(M) \times A_\pi(M)$.
  If $v^l(t(\bar{a}') - nz ) = \gamma$ for all $\bar{a}'$ of the same order type as $\bar{a}$, then this value is a constant.
  If $v^l(t(\bar{a}') - nz ) = v^l(b_0(\bar{a}')-nc_0)$ for all $\bar{a}'$ of the same order type as $\bar{a}$, then this value can always be calculated in $Z_\pi(M)$.
  If we are not in one of these two cases, then the quantifier free 0-definable coloring
  \[ C_{<\gamma}(i) \iff i < \gamma \]
  witnesses (in $Z_\pi(M)$) that $J$ is not indiscernible over $z$. Hence we can work in $Z_\pi(M)$ and therefore we can assume that all the $l_i$ coincide (by \cref{lem:multiply_congruence}).
  Moreover, to simplify the notation we will assume that all the $l_i$ are equal to $1$.

  We say that an indiscernible sequence $J$ has an approximation for $z$ over $\alpha \in I$ if there is a set $D$ such that $J$ is indiscernible over $D$, $\alpha$ is definable over $D$, and the residue class of $z$ modulo $B_\alpha$ is algebraic (in $Z/B_\alpha$) over parameters in $D$.
  
  We now assume that the mutually indiscernible sequences $J_1$ and $J_2$ both fail to be indiscernible over $z$. Then this must be witnessed as in (a) to (d).
  Such a witness for $J_1$ (resp. $J_2$) is \emph{good} if $J_2$ (resp. $J_1$) has an approximation for $z$ for a suitable $\alpha$ defined as follows:
  \begin{itemize}
  \item If the witness is given as in (a), then we set
  \[ \alpha = \max\{ v(t(\bar{a})-t(\bar{a}') ), v(t'(\bar{a}) - t'(\bar{a}')) \} +1. \]
  If (for example) $\alpha = v(t(\bar{a})- t(\bar{a}') ) +1$ ($ = v(t(\bar{a})-nz )+1 < v(t(\bar{a}')-nz )+1$), then $t(\bar{a'}) \equiv n'z \mod B_\alpha$.
  Therefore the residue class of $z$ modulo $B_\alpha$ is algebraic over $t(\bar{a}')$.
  
  \item If the witness is given as in (b), then we set
  \[ \alpha = v(t(\bar{a})-t(\bar{a}') ) + 1 = \min \{ v(t(\bar{a})-nz ), v(t(\bar{a}')-nz ) \} +1. \]
  If $v(t(\bar{a})-nz ) < v(t(\bar{a}')-nz )$, then $t(\bar{a'}) \equiv n'z \mod B_\alpha$ and therefore
  the residue class of $z$ modulo $B_\alpha$ is algebraic over $t(\bar{a}')$.
  
  \item If the witness is given as in (c), we set
  \[ \alpha = v(t(\bar{a})-nz). \]
  
  \item Now assume the witness is given as in (d). We set
  \[ \alpha_1 = v(t(\bar{a})-nz) \quad \text{and} \quad \alpha_2 = v(t'(\bar{a}) -n'z) + 1. \]
  Now put $\alpha = \max \{ \alpha_1, \alpha_2 \}$.
  \end{itemize}
  In particular, every witness of type (a) or (b) is good because $J_1$ and $J_2$ are mutually indiscernible. 
  Recall that if $v(x) < v(y)$, then $v(x-y) = v(x)$.
  We aim to show that we can always find a good witness:

  Suppose the witness is given as in (c). Choose $\bar{a}_0 \subseteq J_1$ of the same order type as $\bar{a}$ and $\bar{a}'$ such that all indices involved in $\bar{a}_0$ are smaller than the indices in $\bar{a}$ and $\bar{a}'$ (from now on, we will write $\bar{a_0} \ll \bar{a},\bar{a}'$ in that case). If $v(t(\bar{a}_0)-nz) \neq v(t(\bar{a}) - nz)$, then either the pair $(\bar{a_0},\bar{a})$ or the pair $(\bar{a}_0,\bar{a}')$ gives a good witness as in case (b).
  
  Hence we will assume $v(t(\bar{a}_0)-nz) = v(t(\bar{a}) - nz)$. Let $J_1^{>\bar{a}_0}$ be the sequence consisting of all elements of $J_1$ with index larger than all indices in $\bar{a}_0$
  and set $J_2\cup\bar{a}_0$ to be the sequence $J_2$ where each tuple is expanded by $\bar{a}_0$. Then $J_1^{>\bar{a}_0}$ and $J_2\cup\bar{a}_0$ are mutually indiscernible.
  Moreover, $J_1^{>\bar{a}_0}$ is not indiscernible over $\alpha = v(t(\bar{a}) - nz)$ (as witnessed by $\bar{a}$ and $\bar{a}'$). Hence $J_2\cup\bar{a}_0$ is indiscernible over $\alpha$ by \cref{lem:dp-min_value}.
  Now $J_2$ is indiscernible over the set $\{ \bar{a}_0, \alpha \}$ and we have $\bar{a}_0 \equiv nz \mod B_\alpha$. Therefore the residue class of $z$ modulo $B_\alpha$ is algebraic over $\bar{a}_0$ and hence the witness is good.
  
  Now suppose the witness is given as in (d). We set
  \[ \alpha_1 = v(t(\bar{a})-nz) \quad \text{and} \quad \alpha_2 = v(t'(\bar{a}) -n'z) + 1. \]
  If $\alpha_2 \geq \alpha_1$, then we have a good witness by the same arguments as in (a) and (b).
  Hence assume $\alpha := \alpha_1 > \alpha_2$. Suppose for all $\bar{a}_0 \ll \bar{a},\bar{a}'$ we have $v(t(\bar{a}_0)-nz) = \alpha$. Fix
  \[ \bar{a}_0 \ll \bar{a}_1 \ll \bar{a}, \bar{a}'. \]
  Consider the mutually indiscernible sequences $J_1^{>\bar{a}_0}$ and $J_2\cup\bar{a_0}$.
  
  Assume that $J_1^{>\bar{a}_0}$ is indiscernible over $\alpha$. Then the residue class of $z$ modulo $B_\alpha$ is algebraic over $t(\bar{a}_1)$.
  Since $t'(\bar{a}) \not \equiv n'z \mod B_\alpha$, we get $t'(\bar{a}') \not \equiv n'z \mod B_\alpha$ by indiscernibility (applied to $\alpha$ and $\bar{a}_1$).
  Therefore $v(t'(\bar{a})-n'z)$ and $v(t'(\bar{a}')-n'z)$ only depend on the residue class of $z$ modulo $B_\alpha$ (and can be calculated in $Z/B_\alpha$) and hence cannot witness the failure of indiscernibility over $z$.
  
  Hence $J_1^{>\bar{a}_0}$ is not indiscernible over $\alpha$. Then $J_2\cup\bar{a}_0$ is indiscernible over $\alpha$ by \cref{lem:dp-min_value}.
  Therefore $J_2$ is indiscernible over $\{\bar{a}_0, \alpha\}$ and the residue class of $z$ modulo $B_\alpha$ is algebraic over $\bar{a}_0$. Hence we have a good witness.
  
  Hence we assume that there is $\bar{a}_0 \ll \bar{a},\bar{a}'$ such that
  \[ v(t(\bar{a}_0) - nz ) \neq \alpha. \]
  If $v(t(\bar{a}_0) - nz ) > \alpha$, then $\alpha = v(t(\bar{a}_0)-t(\bar{a}))$ and we have a good witness as in cases (a) and (b).
  Hence we assume $v(t(\bar{a}_0) - nz ) < \alpha$.
  
  If $v(t'(\bar{a}_0) - n'z) \not \in \{ v(t'(\bar{a})-n'z), v(t'(\bar{a}')-n'z) \}$, then $(\bar{a}_0,\bar{a})$ or $(\bar{a}_0,\bar{a}')$ gives a good witness as in case (a).
  If $v(t'(\bar{a}_0) -n'z) = v(t'(\bar{a})-n'z)$, then either $(\bar{a}_0, \bar{a}')$ gives a witness as in case (a) or the new witness is given by $(\bar{a}_0,\bar{a})$ and we have
  \[ v(t(\bar{a}) - nz ) > v(t(\bar{a}_0) -nz) \quad \text{and} \quad v(t'\bar{a}) - n'z ) = v(t'(\bar{a}_0)-n'z). \]
  Hence we are again in case (d) but $J_2$ is indiscernible over 
  \[ v(t'(\bar{a})-n'z ) = v(t'(\bar{a}) - t'(\bar{a}')) \]
  and hence this witness given by $(\bar{a}_0,\bar{a})$ must be good.
  
  Now only the case $v(t'(\bar{a}_0)-n'z) = v(t'(\bar{a}')-n'z)$ is left. We then have
  \[ v(t(\bar{a})-nz ) = v(t(\bar{a}')-nz) > v(t(\bar{a}_0)-nz ), \]
  \[ v(t'(\bar{a})-n'z) < v(t'(\bar{a}')-n'z) = v(t'(\bar{a_0})-n'z). \]
  
  Assume the witnessing formula was of the form
  \[ R(v(t(\bar{x})-nz), v(t'(\bar{x})-n'z)) \]
  for a monotone binary relation $R$ (the other case is done analogously).
  
  We then have the following implications by monotonicity:
  \begin{align*}
    &\models R(v(t(\bar{a})-nz), v(t'(\bar{a})-n'z)) \\
    \implies &\models R(v(t(\bar{a}')-nz), v(t'(\bar{a}')-n'z)) \\
    \implies &\models R(v(t(\bar{a}_0)-nz), v(t'(\bar{a}_0)-n'z)).
  \end{align*}
  Hence $R(v(t(\bar{a})-nz), v(t'(\bar{a})-n'z))$ must be false and $R(v(t(\bar{a}')-nz), v(t'(\bar{a}')-n'z))$ must be true (since this was a witness for the failure of indiscernibility over $z$).
  Then $R(v(t(\bar{a}_0)-nz), v(t'(\bar{a}_0)-n'z))$ must be true. But then $\bar{a}$ and $\bar{a}_0$ give a witness as in (a). Hence we can always find a good witness.
  
  Since we assume that both $J_1$ and $J_2$ fail to be indiscernible over $z$, we can find a good witness for each of them. Let $\alpha$ be the constant for the witness in $J_1$ and let $\beta$ be the constant for the witness in $J_2$. We assume $\alpha \leq \beta$. Then $J_1$ is indiscernible over $\beta$ and over the residue class of $z$ in $Z/B_\beta$.
  
  Suppose the witness for $J_1$ is given as in (a) or (b). If we have
  \[ v(t(\bar{a})-nz ) < v(t(\bar{a}')-nz ), \]
  then $v(t(\bar{a})-nz ) < \beta$ and indiscernibility (and algebraicity of $z$ modulo $B_\beta$ over a suitable parameter) imply that
  $v(t(\bar{a}')-nz ) < \beta$. Hence those values only depend on the residue class of $z$ modulo $B_\beta$ (and can be calculated in $Z/B_\beta$ with the restricted valuation).
  Therefore they cannot witness the failure of indiscernibility over $z$.
  
  Now suppose the witness for $J_1$ is given as in case (c). If $\alpha = \beta$, then this cannot be a witness for the failure for indiscernibility. Hence we must have $\alpha < \beta$. But then 
  \[ v(t(\bar{a})-nz ) = v(t(\bar{a}')-nz ) \]
  only depends on the residue class of $z$ in $B_\beta$ and we can argue as before.
  The same arguments work if the witness for $J_1$ is given as in case (d).
  
  Hence $J_1$ or $J_2$ must be indiscernible over $z$.   
\end{proof}
  
To characterize distality we will show that the quotients $B_i/B_{i+1}$ are stable. We will make use of the following lemma:
\begin{lemma}[Lemma 5.13 of \cite{alouf-delbee}] \label{lem:stable_reduct}
Let $\calL_0$ be any language and let $T_0$ be an unstable $\calL_0$-theory. Let $\calL_0^- \subseteq \calL_0$ be such that $T_0|_{\calL_0^-}$ is stable. Then there exists an $\calL_0$-formula
$\phi(x,y)$, $|x|=1$, over $\emptyset$ and a parameter $b$ such that $\phi(x,b)$ is not $\calL_0^-$-definable.
\end{lemma}
  
\begin{proposition} \label{prop:induced_stable}
Suppose $B_i/B_{i+1}$ is infinite. Then the induced structure on $B_i/B_{i+1}$ is stable.
\end{proposition}
\begin{proof}
Suppose $B_i/B_{i+1}$ is infinite. By \cref{lem:stable_reduct} it suffices to show that for every formula $\tilde{\phi}(\tilde{x},\tilde{y})$ (in $B_i/B_{i+1}$) and every constant $\tilde{b} \in B_i/B_{i+1}$ the formula $\tilde{\phi}(\tilde{x},\tilde{b})$ is definable in the pure group $(B_i/B_{i+1},+)$.

Given such a formula there is an $\calL_\text{mon}^S$-formula $\phi(x,y)$ such that $\phi$ is the preimage of $\tilde{\phi}$ under the natural projection
\[ \pi_i : B_i \rightarrow B_i/B_{i+1}. \]
Now fix a preimage $b$ of $\tilde{b}$. Note that $\phi(B_i,b)$ is a union of cosets of $B_{i+1}$.

By the quantifier elimination result $\phi$ is equivalent to a boolean combination of atomic $\calL_\text{mon}^S$-formulas.
We aim to show the following:

\begin{claim}
There is a formula $\psi(x,y)$ which is defined in the pure abelian group $(B_i,+)$ such that $\phi(B_i,b)$ and $\psi(B_i,b)$ coincide on all but finitely many cosets of $B_{i+1}$.
\end{claim}
It suffices to prove the claim for atomic formulas. Therefore we may assume that $\phi$ is atomic. Then we are in one of the following cases:
\begin{enumerate}
\item[(a)] $\phi(x,b) \equiv nx-t(b) = 0$,
\item[(b)] $\phi(x,b) \equiv R(v^{l_1}(n_1x-t_1(b)),v^{l_2}(n_2x-t_2(b)))$,
\item[(c)] $\phi(x,b) \equiv C(v^l(nx-t(b)))$.
\end{enumerate}
In case (a) there is nothing to show. Therefore we consider the cases (b) and (c) which include valuations.
We show that sets of the form
\[ (a+lB_j) \cap B_i \]
are definable in the pure abelian group $(B_i,+)$ up to a unique coset of $B_{i+1}$:

If $j<i$, then $B_j > B_i$ and $lB_j$ has finite index in $B_j$. Hence $lB_j \cap B_i$ has finite index in $B_i$. Moreover, $lB_j \cap B_i$ is of the form
\[ lB_j \cap B_i = l'B_i \]
for a positive integer $l'$ because this holds true for the standard models. The cosets of $l'B_i$ are definable in the pure group language.
If $j=i$, then $lB_i \cap B_i = lB_i$ is definable in the pure group language.
Now assume $j>i$. Then $lB_j < B_j \leq B_{i+1}$ and therefore $B_i \cap (a+lB_j)$ is trivial outside of a single coset of $B_{i+1}$.

This also shows that there are only finitely many intersections of the form
\[ (a+lB_j) \cap (B_i \setminus (a+B_{i+1})) \]
where everything except $j$ is fixed. Therefore the restriction of $v^l(x-a)$ to $B_i \setminus (a+B_{i+1})$
is given by a finite chain of definable subgroups (in the pure abelian group $(B_i,+)$).

Since $nx \equiv 0 \mod B_{i+1}$ has only finitely many solutions modulo $B_{i+1}$ the same holds true for the valuation $v^l(nx-a)$ restricted to $B_i$: Outside of finitely many cosets of $B_{i+1}$ it is given by a finite chain of $(B_i,+)$-definable subgroups. In that sense $v^l(nx-a)$ is $(B_i,+)$-definable outside of finitely many cosets of $B_{i+1}$.

Therefore the formula $\phi$ in (a) or (b) is definable in $(B_i,+)$ outside of finitely many cosets of $B_{i+1}$. This shows the claim.

Hence we can find such a formula $\psi(x,y)$ defined in the pure abelian group $(B_i,+)$ such that $\phi(B_i,b)$ and $\psi(B_i,b)$ coincide on all but finitely many cosets of $B_{i+1}$.
The usual quantifier elimination result for abelian groups shows that $\psi(B_1,b)$ is a boolean combination of cosets of the trivial subgroup and groups of the form $lB_i$ for $l \geq 1$.
Each subgroup $lB_i$ has finite index in $B_i$ and the family $\{lB_i : l \geq 1 \}$ is closed under finite intersections. Hence a boolean combination of such groups is a union of finitely many cosets of $lB_i$ for a suitable $l$.

Since $\phi(B_i,b)$ and $\psi(B_i,b)$ agree on all but finitely many cosets of $B_{i+1}$ and $\phi(B_i,b)$ is a union of cosets of $B_{i+1}$, the same must be true for $\tilde{\phi}(B_i/B_{i+1},\tilde{b})$, i.e. $\tilde{\phi}(B_i/B_{i+1},\tilde{b})$ is a boolean combination of cosets of $(B_i/B_{i+1},+)$-definable subgroups. Therefore $\tilde{\phi}(B_i/B_{i+1},\tilde{b})$ is definable in $(B_i/B_{i+1},+)$.
\end{proof}
  
\begin{theorem} \label{thm:distality}
  $T$ is distal if and only if there is a constant $k < \omega$ such that
  \[ |B_i / B_{i+1}|  \leq k \]
  holds for all $i < \infty$.
\end{theorem}
\begin{proof}
  Suppose $|B_i/B_{i+1}|$ is unbounded. Then there is some $i_0$ such that $|B_{i_0}/B_{i_0+1}|$ is infinite. By \cref{prop:induced_stable} the induced structure
  on $B_{i_0}/B_{i_0+1}$ is stable. Hence it follows from \cref{distal_eq} that $T$ is not distal.
  
  Now let $X$ be a non-constant totally indiscernible set of singletons and fix $x, y \in X$. Put $i_0 = v(x^{-1}y)$. If $x \neq y$, then $i_0 < \infty$ and hence
  \[ xB_{i_0} = yB_{i_0} \quad \text \quad xB_{i_0+1} \neq xB_{i_0+1}. \]
  It follows easily from total indiscernibility that $i_0$ does not depend on the choice of $x \neq y$. Hence $|X| \leq |B_{i_0}/B_{i_0+1}|$. 
\end{proof}

\section{Valuations on the integers} \label{ch:integers}
The most well-known example of a dp-minimal expansion of $(\bbZ,+)$ is $(\bbZ,+, \leq)$. Based on work by Palac\'in and Sklinos \cite{palacin-sklinos}, Conant and Pillay \cite{conant-pillay} proved the remarkable result that $(\bbZ,+,0)$ has no proper stable expansions of finite dp-rank. Hence any proper dp-minimal expansion must be unstable.
The other known examples of dp-minimal expansions are:
\begin{itemize}
\item $(\bbZ,+,v_p)$ where $v_p$ is the $p$-adic valuation on $\bbZ$. This was shown by Alouf and d'Elb\'ee in \cite{alouf-delbee}.
\item $(\bbZ,+,C)$ where $C$ is cyclic order. These were found by Tran and Walsberg in \cite{tran-walsberg}.
\item Proper dp-minimal expansions of $(\bbZ,+,S)$, where $S$ is a dense cyclic order, and $(\bbZ,+,v_p)$ were very recently found by Walsberg in \cite{walsberg}.
\end{itemize}
An overview about the current research on dp-minimal expansions of $(\bbZ,+)$ is given by Walsberg in Section 6 of \cite{walsberg}.

\subsection{A single valuation}
We add the following family of examples which generalize the $p$-adic examples by Alouf and d'Elb\'ee:

\begin{theorem}\label{thm:valuation_on_Z}
  Let $(B_i)_{i<\omega}$ be a strictly descending chain of subgroups of $\bbZ$, $B_0 = \bbZ$, let $v: \bbZ \rightarrow \omega \cup \{\infty\}$ be the valuation defined by
  \[ v(x) = \max\{ i : x \in B_i \}, \]
  and let $S$ be a set of unary predicates and monotone binary relations on the value set.
  Then $(\bbZ,0,1,+,v,S)$ admits quantifier elimination in the language $\calL_\mathrm{mon}^S$ (with $0$ and $1$ as constants) and is dp-minimal. Moreover, $(\bbZ,0,1,+,v,S)$ is distal if and only if the size of the quotients $B_i/B_{i+1}$ is bounded.
\end{theorem}
\begin{proof}
  Note that any infinite strictly descending chain of subgroups of $\bbZ$ must have trivial intersection. Moreover, every non-trivial subgroup of $\bbZ$ is of the form $n\bbZ$ for some $n \geq 1$ and hence $v$ is a good valuation in the sense of \cref{good_valuation}.
  
  Moreover, $\bbZ \prec \hat{\bbZ} = \prod_p \bbZ_p$ and $\langle 1 \rangle = \bbZ$ is dense. Hence \cref{qe_monotone} implies the quantifier elimination result. Dp-minimality follows by \cref{thm:dp-minimality} and the claim about distality follows by \cref{thm:distality}.
\end{proof}

In case of the $p$-adic valuation Alouf and d'Elb\'ee proved in Theorem 1.1 of \cite{alouf-delbee} that $(\bbZ, +,  v_p)$ has quantifier elimination in the language $\calL_p^E = \{+,-,0,1,|_p,D_n)_{n \geq 1}$ where
\[x|_py \iff v_p(x) \leq v_p(y) \quad \text{and} \quad D_n = n\bbZ. \]

Conant \cite{conant} showed that the structure $(\bbZ,+,0,1,\leq)$ is a minimal proper expansion of $(\bbZ,+,0,1)$, i.e. there is no proper intermediate expansion.
Alouf and d'Elb\'ee proved the same for $(\bbZ,+,0,1,v_p)$.  We will show that this does not hold true for arbitrary valuations.

\begin{proposition} \label{valuation_retract}
  Fix distinct primes $p_0,p_1,q \in \bbP$ and put $s=p_0p_1q$. For $i<\omega$ fix $\sigma_i \in \Sym(\{0,1\})$, set $n_0 = 1$, and recursively define
  \[ n_{3l+m} = \begin{cases} n_{3l-1}q & \text{ iff } m = 0, \\ n_{3l}p_{\sigma_l(0)} & \text{ iff } m = 1, \\ n_{3l+1}p_{\sigma_l(1)} & \text{ iff } m = 2. \end{cases} \]
  Set $v_\sigma$ to be the valuation corresponding to $(n_i\bbZ)_{i<\omega}$ and let  $w$ be the valuation corresponding to $(s^i\bbZ)_{i< \omega}$.
  Then $w$ is definable in $(\bbZ,+, v_\sigma)$.
\end{proposition}
\begin{proof}
  If $a \in \bbZ \setminus \{ 0 \}$, then there is a unique $t_a \in \{ p_0,p_1,q\}$ such that $v_\sigma(a) < v_\sigma(t_a a)$.
  Let $a,b \in \bbZ \setminus\{ 0 \}$. If $|v_\sigma(a)-v_\sigma(b) | \geq 3$, then 
  \[ w(a) < w(b) \iff v_\sigma(a) < v_\sigma(b). \]
  If $|v_\sigma(a)-v_\sigma(b) | < 3$, then $w(a) \leq w(b)$ can be determined using $t_a$ and $t_b$.
\end{proof}

\begin{corollary}
Let $w$ be as in \cref{valuation_retract}. Then there are $2^{\aleph_0}$ many valuations $v$ such that $w$ is definable in $(\bbZ,+,v)$. Only countably many of those can be definable in $(\bbZ,+,w)$.
\end{corollary}
\begin{proof}
There are $2^{\aleph_0}$ many valuations $v_\sigma$ as in \cref{valuation_retract} and $w$ is definable in each $(\bbZ,+,v_\sigma)$. On the other hand, $(\bbZ,+,w)$ has only countably many definable sets.
\end{proof}

\begin{remark}
  Note that by \cref{thm:distality} all these structures are distal. Hence not even all dp-minimal distal expansions by valuations are minimal expansions. 
\end{remark}

The fact that expansions by arbitrary valuations are dp-minimal allows us to construct other non-trivial examples:
For $k \geq 2$ let $v_k$ denote the valuation corresponding to the sequence $(k^i\bbZ)_{i < \omega}$.
\begin{proposition}
  Let $r$ and $s$ be coprime positive integers. Then the expansion 
  \[(\bbZ,+,v_r(x)< v_s(x))\]
  is dp-minimal.
\end{proposition}
\begin{proof}
  We have
  \[ v_r(x) < v_s(x) \iff v_{rs}(x) < v_{rs}(rx). \]
  Hence the relation $v_r(x) < v_s(x)$ is definable in the dp-minimal structure $(\bbZ,+,v_{rs})$.
\end{proof}

It seems unlikely that $v_{rs}$ is definable from $v_r(x) \leq v_s(x)$.

The induced structure on the index set $\omega \cup \{\infty \}$ seems to be important. If it is not o-minimal and $X \subseteq \omega \cup \{ \infty \}$ is a definable infinite and co-infinite subset, then the set
\[ A = \{ a \in \bbZ : w(a) \in X \} \subseteq \bbZ \]
is definable. It is not clear if $w$ is definable in $(\bbZ,+,0,1,A)$.

If the induced structure on $\omega \cup \{ \infty \}$ is o-minimal, then $k = |B_i/B_{i+1}| \in \bbN \cup \{ \infty \}$ must be constant for all sufficiently large $i$ (in some elementary extension). If $k$ is finite, then the size of the quotients $B_i/B_{i+1}$ is bounded and hence we are in the distal case.

\begin{conjecture}  \label{conj:distality}
  Let $(\bbZ,+,0,1,v)$ be distal. Then the following are equivalent:
  \begin{enumerate}
  \item[(a)] $(\bbZ,+,0,1,v)$ is a minimal expansion of $(\bbZ,+,0,1)$,
  \item[(b)] there is a prime $p$ such that $|B_i/B_{i+1}| = p$ for almost all $i < \omega$,
  \item[(c)] $v$ is interdefinable with a $p$-adic valuation for some prime $p$,
  \item[(d)] the $(\bbZ,+,v)$-induced structure on the value set of $v'$ is o-minimal for all $(\bbZ,+,v)$-definable valuations $v'$.
  \end{enumerate}
\end{conjecture}

\begin{proposition}
 If (a) implies (d), then \cref{conj:distality} holds.
\end{proposition}
\begin{proof}
  We already know (b) $\implies$ (c) $\implies$ (a) and by assumption (a) $\implies$ (d) holds.
  Hence (d) $\implies$ (b) remains to be shown.
  
  Let $(\bbZ,+,0,1,v)$ be distal and assume (d). Then there is $k > 1$ such that
  \[ |B_i:B_{i+1}| = k \]
  for almost all $i < \omega$. Therefore $v$ and $v_k$ are interdefinable and we may assume $v = v_k$.
  
  If $k = st$ where $s$ and $t$ are coprime, then $v$ is interdefinable with the valuation $w$ such that $|B_i^{w}/B_{i+1}^{w}|$ alternates between $s$ and $t$. Then the induced structure on the value set of $w$ is not o-minimal. This contradicts (d).
  
  Hence we may assume $k = p^n$ for some prime $p$ and $n \geq 1$.
  If $n > 1$, then the $p$-adic valuation $v_p$ is definable by
  \[ v_p(x) \leq v_p(y) \iff \bigwedge_{r=0}^{n-1} v_{p^n}(p^rx) \leq v_{p^n}(p^ry). \]
  Now the set $v_p(\{a \in \bbZ : v_{p^n}(pa) > v_{p^n}(a) \})$ contradicts (d).
  
  Hence $k = p$ must be a prime. This shows (b).
\end{proof}

There are non-distal candidates for minimal expansions:
\begin{question}
Let $(p_i)_{0 < i < \omega}$ be an enumeration of the primes such that each prime appears exactly once and let $v$ be the valuation corresponding to $(p_1\cdots p_i\bbZ)_{i<\omega}$. 
Then the induced structure on the value set is o-minimal. Is $(\bbZ,+,0,1,v)$ a minimal expansion of $(\bbZ,+,0,1)$?
\end{question}

We end this section with the observation that the $p$-adic valuations have a limit theory:

\begin{proposition} \label{prop:limit_theory}
  For each prime $p$ let $v_p$ denote the $p$-adic valuation on $\bbZ$. Then the corresponding limit theory exists, i.e.
  \[ \Th( \prod_p(\bbZ,+,v_p) / \calU ) \]
  does not depend on the choice of the non-principal ultrafilter $\calU \subseteq \calP(\bbP)$.
\end{proposition}
\begin{proof}
  We fix the common $\calL^-$-theory
  \[ T = \bigcap_\calU \Th_{\calL^-}( \prod_p(\bbZ,+,v_p) / \calU  ) \]
  of these ultraproducts. Note that the predicate $\Div_{q^k}^l(i,j)$ fails for all $0<i<j$ and the predicate $\Ind_k^l$ holds true for all $0<i<j$. Thus they are quantifier free 0-definable after naming the successor function $S$ on $\calI$.
  Therefore $T$ has quantifier elimination after naming $S$ by \cref{qe_theorem} (because $(\omega,0,\leq,S)$ has quantifier elimination).
  The constants in $\calL_Z$ generate a subgroup that is isomorphic to $\bbZ$. An element $a \in \bbZ \setminus \{0\}$ must have valuation $0$ in all models of $T$.
  Therefore all models of $T$ have isomorphic substructures and hence $T$ is complete.
\end{proof}

\subsection{Multiple valuations}
If $P \subseteq \bbP$ is a non-empty set of primes, then Alouf and d'Elb\'ee proved that the structure $(\bbZ,+,v_p)_{p\in P}$ has dp-rank exactly $|P|$.
We will generalize this result to expansions of $(\bbZ,+)$ by arbitrary valuations which involve disjoint sets of primes.

Let $V$ be a non-empty family of non-trivial valuations $v:\bbZ \rightarrow \omega \cup \{ \infty \}$. For each $v \in V$ set
\[ \pi_v = \{ p \in \bbP : p \text{ divides } |B_i^v/B_{i+1}^v| \text{ for some } i < \omega \}. \]
We view $(\bbZ,+)$ together with these valuations as a multi-sorted structure with group sort $\calZ$ and with a distinct value sort
$\calI_v$ for each valuation $v \in V$. 
Now put
\begin{itemize}
\item $\calL_\calZ = \{ 0, 1, +, - \}$,
\item $\calL_v = \{ v^l : l \geq 1 \}$, and
\item $\calL_{\calI_v}^- = \{ -\infty, 0, +\infty, \leq, \Div_{q^k}^l, \Ind_k^l \}_{q,k,l}$
\end{itemize}
for each valuation $v \in V$. Let $\calL_\text{mon}^v$ be the monotone hull of $\calL_{\calI_v}^-$ as in \cref{sec:monotone_hull} and set
\[ \calL_\text{mon} = \calL_\calZ \cup (\bigcup_{v \in V} \calL_v ) \cup (\bigcup_{v \in V} \calL_\text{mon}^v ) \]
to be the disjoint union of these languages.

\begin{proposition}
Suppose the sets $\pi_v$ are pairwise disjoint. Then $(\bbZ,+,v)_{v\in V}$ has quantifier elimination in the language $\calL_\text{mon}$.
\end{proposition}
\begin{proof}
This is very similar to the proof of \cref{qe_theorem}.
Note that a multi-sorted version of \cref{lem:qe} holds true in this setting. As in the proof of \cref{qe_theorem} it suffices to show the back-and-forth property for systems of linear congruences.
Let $\calS$ be the system
\begin{align*}
n_sx - b_s & \equiv 0 \mod l_sB_{i_s}^v, \quad s \in S^v, \\
n_tx - b_t & \not \equiv 0 \mod l_tB_{i_t}^v, \quad t \in T^v,
\end{align*}
where $S^v \neq \emptyset$ and $T^v$ are finite index sets for each $v \in V_0$ for a finite subset $V_0 \subseteq V$. 
By an application of \cref{lem:multiply_congruence} we may assume that all the $l_s$ and $l_t$ have the same value which we denote by $l$.

Let $a$ be a solution.
We will show that we can assume that $a$ and all constants $b_s$ and $b_t$ are contained in $lB_0$:
If $a$ is in $lB_0$, then all $b_s$ must be contained in $lB_0$ since otherwise the congruences can not be satisfied.
If $b_t$ is not contained in $lB_0$, then the congruence
\[ n_tx - b_t  \not \equiv 0 \mod lB_{i_t}^v \]
does not impose any restrictions on $lB_0$ and we can ignore it without changing the solutions in $lB_0$.

If $a$ is not contained in $lB_0$, then there is a constant $c \in \bbZ$ (and hence in the language) such that $a-c \in lB_0$. In that case
the shifted system $\calS^c$:
\begin{align*}
n_s(x+c) - b_s & \equiv 0 \mod lB_{i_s}^v, \quad s \in S^v, \\
n_t(x+c) - b_t & \not \equiv 0 \mod lB_{i_t}^v, \quad t \in T^v,
\end{align*}
is solved by $a-c \in lB_0$ and all the constants $n_sc-b_s$ and $n_tc-b_t$ can be assumed to lie in $lB_0$. Thus we can replace $\calS$ by $\calS^c$.

Hence we may assume that $\calS$ is a system of linear congruences in the subgroup $lB_0$. We have
\[ lB_0 \equiv \hat{\bbZ} = \prod_p \bbZ_p \]
and the valuations $v^l$ involve disjoint sets of primes. Therefore the system $\calS$ can be solved independently for each valuation $v\in V$. This is done as in the proof of \cref{qe_theorem}.
\end{proof}

\begin{theorem}
Suppose the sets $\pi_v$ are pairwise disjoint. Then 
\[ \dprk((\bbZ,+,v)_{v \in V}) = |V|. \] 
\end{theorem}
\begin{proof}
$\geq$ is shown exactly as in the case of the $p$-adic valuations which was done by Alouf and d'Elb\'ee (Theorem 1.2 of \cite{alouf-delbee}).

Now assume $\kappa := \dprk((\bbZ,+,v)_{v \in V}) > |V|$. As in the proof of \cref{thm:dp-minimality} this is witnessed by mutually indiscernible sequences 
$(I_i)_{i<\kappa}$ (in the group sort) and a singleton $c$ in the group sort such that no sequence is indiscernible over $c$. As argued in \cref{thm:dp-minimality}, the fact that a sequence $I$ is not indiscernible over $c$ must be witnessed by an atomic $\calL_\text{mon}$-formula which involves a valuation.

Since $\kappa > |V|$, there must be two sequences $I_1$ and $I_2$ for which this witnessing formula involves the same valuation $v$. This is a contradiction because $(\bbZ,+,v)$ is dp-minimal by \cref{thm:dp-minimality}.
\end{proof}

\section{Further results} \label{ch:further}

\subsection{Uniformly definable families of finite-index subgroups of dp-minimal groups} \label{sec:uniform_families}
The classification of NIP profinite groups by Macpherson and Tent in \cite{macpherson-tent} yields information about uniformly definable families of finite index subgroups in arbitrary NIP groups (see Theorem 8.7 in \cite{some_model_theory}). We will do the same in the dp-minimal case. The arguments are almost identical to those in Section 8 of \cite{some_model_theory} (see also Remark 5.5 in \cite{macpherson-tent}), we only need to make sure that the construction presented there preserves dp-minimality.

Let $H$ be a group and let $(N_i : i \in I)$ be a family of normal subgroups of finite index such that
\[ \forall i,j \exists k : N_k \leq N_i \cap N_j. \]
We view $H$ as an $\mathcal{L}_\text{prof}$-structure $\mathcal{H} = (H,I)$. Let $f_j : \varprojlim H/N_i \rightarrow H/N_j$ be the projection maps. Then $\{ \ker f_j : j \in I \}$ is a neighborhood basis at the identity. Therefore we may view $\varprojlim H/N_i$ as an $\mathcal{L}_\text{prof}$-structure $(\varprojlim H/N_i, I)$.

\begin{lemma} \label{lem:profinite_quotient}
	Let $\mathcal{H}^* = (H^*,I^*)$ be an $|I|^+$-saturated elementary extension of $\mathcal{H}$. Then 
	\[ \mathcal{H}^* / \bigcap_{i \in I} N_i^* \cong \varprojlim_{i \in I} H/N_i \]
	and $(N_j^* / \bigcap_{i \in I} N_i^* : j \in I )$ is a neighborhood basis for the identity consisting of open normal subgroups.
\end{lemma}
\begin{proof}
	By elementarity we have $|H^*:N_i^*| = |H:N_i|$ for all $i \in I$. Using this and elementarity it is easy to see that
	\[ \varprojlim_{i \in I} H^*/N_i^* = \varprojlim_{i \in I} H/N_i. \]
	Now write
	\[ \varprojlim_{i \in I} H^*/N_i^* = \{ (g_iN_i^*)_i : \forall i \geq j: g_iN_j^* = g_jN_j^* \} \]
	and let $f : H^* \rightarrow  \varprojlim_{i \in I} H^*/N_i^*, g \mapsto (gN_i^*)_i$ be the natural homomorphism.
	Clearly $\ker f = \bigcap_{i \in I} N_i^*$. It remains to show that $f$ is surjective.
	
	Fix $(g_iN_i)_i \in \varprojlim_{i \in I} H^*/N_i^*$ and consider the partial type
	\[ \Sigma(x) = \{ x \in g_iN_i^* : i \in I \}. \]
	Given $I_o \subseteq I$ finite, there is $j \in I$ such that
	$N_j \leq \bigcap_{i \in I_0} N_i$. Then $g_iN_j = g_{i'}N_j$ for all $i,i' \in I_0$.
	Hence $\Sigma(x)$ is finitely satisfiable and as $\mathcal{H}^*$ is $|I|^+$-saturated, there exists $g \in H^*$ such that $g \in g_iN_i^*$ for all $i \in I$ and hence $f(g) = (g_iN_i^*)$.
	
	The family $(N_j^*/\bigcap_{i \in I} N_i^* : j \in I)$ is a neighborhood basis for the identity consisting of open normal subgroups.
\end{proof}

\begin{lemma}\label{limitlemma}
	If $\mathcal{H} = (H,I)$ has NIP, then $(\varprojlim_{i \in I} H/N_i, I)$ has NIP. If moreover $(H,I)$ is dp-minimal, then $(\varprojlim_{i \in I} H/N_i, I)$ is dp-minimal.
\end{lemma}
\begin{proof}
	Since $(H,I)$ has NIP, every uniformly definable family of subgroups contains only finitely many subgroups of each finite index. Let $(H^*,I^*)$ be an $|I|^+$-saturated elementary extension. Then $I$ is externally definable (since
	$I = \{ i \in I^* : |H^*:K_i^*| < \infty \}$). If $(H,I)$ is dp-minimal, then $(H^*,I^*,I)$ is dp-minimal by \cref{rem:externally_remark}. By the above lemma the structure $(\varprojlim_{i \in I} H/N_i, I)$
	is interpretable as a quotient in $(H^*,I^*,I)$ and hence is NIP (resp. dp-minimal).
\end{proof}

Let $G$ be a group and let $\phi(x,y)$ be a formula. Set $\mathcal{N}_\phi = \{N_i:i\in I\}$ to be the family of all normal subgroups which are finite intersections of conjugates of $\phi$-definable subgroups of finite index.
Note that every $\phi$-definable subgroup of finite index contains some $N \in \mathcal{N}_\phi$.
The profinite group $\varprojlim_{i \in I} G/N_i$ naturally becomes an $L_\text{prof}$-structure $\mathcal{G}_\phi = (\varprojlim_{i \in I} G/N_i, I)$.

\begin{proposition}\label{phicompletion}
	Let $G$ and $\phi$ be as above. If $G$ is NIP, then $\mathcal{G}_\phi$ is NIP. If moreover $G$ is dp-minimal, then $\mathcal{G}_\phi$ is dp-minimal in the group sort.
\end{proposition}
\begin{proof}
	By Baldwin-Saxl finite intersections of conjugates of $\phi$-definable subgroups are uniformly definable by some formula $\psi(x,z)$. The set $J = \{ b : \psi(G,b) \trianglelefteq G \}$ is definable. Put $J_0 = \{ b: |G:\psi(G,b)| < \infty \}$. Then $\mathcal{N}_\phi = \{ \psi(G,b) : b \in J_0 \}$. Since $\mathcal{N}_\phi$ is closed under intersections, it follows that $J_0$ is externally definable. Let $E$ be the equivalence relation defined by $a E b \iff \psi(G,a) = \psi(G,b)$.
	Now apply the previous lemma to the structure $(G,J_0/E)$.
\end{proof}

By \cref{prop:virtual_abelian} every dp-minimal profinite group $(G,I)$ has an open abelian subgroup.
Now \cref{phicompletion} implies the following:

\begin{proposition} \label{dp-min_uniform_family}
	Let $G$ be a dp-minimal group and let $\phi(x,y)$ be a formula. Let $\mathcal{N}_\phi$ be the family of all normal subgroups which are finite intersections of conjugates of $\phi$-definable subgroups of finite index.
	If $\mathcal{N}_\phi$ is infinite, then there is $N \in \mathcal{N}_\phi$ such that for all $M \in \mathcal{N}_\phi$ the quotient $N/(N\cap M)$
	is abelian. 
\end{proposition}
\begin{proof}
  The profinite group $(\varprojlim_{N \in \calN_\phi} G/N, \calN_\phi)$ is dp-minimal and therefore is virtually abelian by \cref{prop:virtual_abelian}.
  Since the quotients $G/N$ are preserved, this implies the proposition.
\end{proof}

\begin{remark}
  By \cref{thm:dp-min_structure} there are essentially two types of dp-minimal profinite groups. This will also be seen in the abelian quotients in the statement of \cref{dp-min_uniform_family}. 
\end{remark}

\begin{remark}
  By Proposition 5.1 of \cite{macpherson-tent} every profinite NIP group $(G,I)$ has an open prosolvable subgroup. Hence if we only assume NIP in the previous theorem, the quotients will be solvable instead of abelian (see Theorem 8.7 of \cite{some_model_theory}).
\end{remark}

\subsection{Strong homogeneity of profinite groups} \label{sec:homogeneity}
Jarden and Lubotzky \cite{jarden-lubotzky} showed that two elementarily equivalent profinite groups are isomorphic if one of them is finitely generated. This was generalized to strongly complete profinite groups by Helbig \cite{helbig}. A profinite group is strongly complete if all subgroups of finite index are open.
The tools used by Helbig and the construction in \cref{sec:uniform_families} give a proof for strong homogeneity.

Let $G$ be a profinite group and suppose $(N_i : i \in I)$ is a neighborhood basis at the identity consisting of open normal subgroups. Let $\calL_P$ be the group language expanded by a family of unary predicates $(P_i : i \in I)$.
We consider $G$ as an $\calL_p$ structure by setting $P_i(G) = N_i$. Note that if $G^*$ is an elementary expansion, then there is a natural $\calL_P$-structure on the quotient $G^*/(\bigcap_{i \in I} P_i(G^*))$.

\begin{lemma} \label{retract}
	Let $G$ a profinite group equipped with an $\calL_P$ structure as above. Let $G^*$ be an elementary extension of $G$ in the language $\calL_P$. Then the composition
	\[ G \rightarrow G^* \rightarrow G^*/(\bigcap_{i \in I} P_i(G^*)) \]
	is an $\calL_P$-isomorphism. 
\end{lemma}
\begin{proof}
  The lemma follows from the same arguments as \cref{lem:profinite_quotient}.
\end{proof}

\begin{proposition} \label{extension}
	Let $G$ and $H$ be profinite groups as $\calL_P$ structures such that the predicates $(P_i : i \in I)$ encode neighborhood bases at the identity consisting of open normal subgroups in both groups. Suppose $A \subseteq G$ is a subset and $f : A \rightarrow H$ is an elementary map with respect to the language $\calL_P$. Then $f$ extends to an $\calL_P$-isomorphism between $G$ and $H$.
\end{proposition}
\begin{proof}
	Let $G^*$ be a common strongly $|A|^+$-homogeneous elementary extension of $G$ and $H$. We can find $\tilde{f} \in \Aut(G^*)$ such that $\tilde{f} |_A = f$. Since $\tilde{f}$ is an $L_P$-automorphism, it induces an automorphism of $G^*/(\bigcap_{i \in I} P_i(G^*))$. Now use \cref{retract} to get the desired isomorphism between $G$ and $H$.
\end{proof}

The following observation in Remark 3.12 in \cite{helbig} is a consequence of Theorem 2 in \cite{smith-wilson} and Corollary 52.12 in \cite{Neumann}:
\begin{theorem} \label{strongly_complete}
	Let $G$ be a profinite group. Then the following are equivalent:
	\begin{itemize}
		\item[(a)] $G$ is strongly complete.
		\item[(b)] For each finite group $A$ there exists a group word $w$ such that $w(A) = 1$ and $w(G)$ is open in $G$.
	\end{itemize}
\end{theorem}

Recall that a group word has \emph{finite width} if $\langle w(G) \rangle = w(G)^n$ for some $n > \omega$. We will make use of the following result:
\begin{proposition}[Proposition 5.2(b) of \cite{wilson}] \label{finite_width}
	Let $G$ be a profinite group. If $w$ is a group word, then $w(G)$ is closed in $G$ if and only if $w$ has finite width in $G$.
\end{proposition}

\begin{proposition}
	Let $G$ and $H$ be profinite groups. Let $A \subseteq G$ be a subset of $G$ and let $f: A \rightarrow H$ be an elementary map. If one of the groups is strongly complete, then $f$ extends to an isomorphism.
\end{proposition}
\begin{proof}
	By \cref{strongly_complete} and \cref{finite_width} strong completeness is a first-order property among profinite groups.
	For each finite group $A$ there is a group word $w_A$ such that $w_A(A) = 1$, $w_A(G)$ is open in $G$, and $w_A(H)$ is open in $H$. Note that by \cref{finite_width} and elementary equivalence of $G$ and $H$, $w_A(G)$ and $w_A(H)$ are definable by the same formula without parameters.
	
	If $N$ is an open normal subgroup of $G$ then $w_{G/N}(G/N) = 1$ and hence $w_{G/N}(G) \subseteq N$. Therefore the family $(w_B(G) : B \text{ a finite group} )$ is a neighborhood basis at the identity.
	
	Hence we may consider $G$ and $H$ as $\calL_P$-structures where the predicates are given by $P_B(G) = w_B(G)$. By \cref{extension} $f$ extends to an isomorphism.
\end{proof}

\subsection{A result on families of subgroups of NTP$_2$ groups} \label{sec:ntp2}
By Theorem 1.1 of \cite{macpherson-tent} a full profinite group is NIP if and only if it is NTP$_2$. Since the structure of these groups is determined by the lattice of subgroups,
this only depends on a single formula. We will show a version for formulas in NTP$_2$ groups.
We will use the following lemma by Macpherson and Tent on groups in NTP$_2$: 

\begin{lemma}[Lemma 4.3 in \cite{macpherson-tent}]\label{lem:NTP2}
Let $G$ be an $\emptyset$-definable  group in a structure with  NTP${}_2$ theory, and $\phi(x,\bar{y})$ a formula implying $x\in G$. Then there is  $k=k_\phi\in {\mathbb N}$ such that the following holds. 
Suppose that $H$ is a subgroup of $G$, $\pi: H\longrightarrow \Pi_{i\in J} T_i$ is an epimorphism to the Cartesian product of the groups $T_i$, and
$\pi_j: H\longrightarrow T_j$ is for each $j\in J$ the composition of $\pi$ with the canonical projection $\Pi_{i\in J} T_i\to T_j$.
Suppose also that for each $j\in J$,  there is a subgroup $\bar{R}_j\leq G$  and group  $R_j <T_j$  with
$\bar{R}_j\cap H=\pi_j^{-1}(R_j)$, such that finite intersections of the groups $\bar{R}_j$ are uniformly definable by instances of $\phi(x,\bar{y})$.
Then $|J|\leq k$.
\end{lemma}

\begin{proposition} \label{ntp2_nip_transfer}
Let $G$ be an NTP$_2$ group and let $\phi(x,y)$ be a formula such that $|x| = 1$.
Suppose that the family $\{ \phi(G,b) : b \in G \}$ consists of normal subgroups of $G$ and is closed under finite intersections.
Then $\phi(x,y)$ has NIP.
\end{proposition}
\begin{proof}
  We aim to show that $\phi$ satisfies the Baldwin-Saxl condition. Let $N_1, \dots N_n$ be instances of $\phi$ and fix $k_\phi$ as in \cref{lem:NTP2}.
  Now set $C_k = \bigcap \{N_i : 1 \leq i \leq n, i \neq k \}$
  and $C = \bigcap \{ N_i : 1 \leq i \leq n \}$. 
  Note that $C_i \cap N_i = C$ and we have $C_i \subseteq N_j$ if $i \neq j$.
  Now set
  \[ H = \langle C_i : 1 \leq i \leq n \rangle. \]
  We then have
  \[ H/C \cong \prod_{i = 1}^n C_i / C. \]
  Now set $T_i = C_i / C$ and assume that all $T_i$ are non-trivial.
  Let 
  \[ \pi : H \rightarrow H/C = \prod_{i=1}^nT_i \]
  be the natural projection and let $\pi_j:H \rightarrow T_j$ be the projection on $T_j$.
  Set $R_j = 1 < T_j$ and put $\bar{R}_j = N_j$.
  Then
  \[ \bar{R}_j \cap H = N_j \cap H = \pi_j^{-1}(R_j). \]
  Hence \cref{lem:NTP2} implies $n \leq k_\phi$.
  
  If $n > k_\phi$, then there must be $1 \leq i \leq n$ such that $T_i = 1$ and hence $C = C_i$ can be written as in intersection of $n-1$ instances of $\phi$. Inductively this shows that any intersection of instances of $\phi$ is an intersection of at most $k_\phi$ instances. Hence $\phi$ satisfies the Baldwin-Saxl condition.
  
  This implies that $\phi$ has NIP: Otherwise we can find a constant $a$ and an indiscernible sequence $(b_i)_{i < \omega}$ such that
  \[ \models \phi(a,b_i) \iff i \text{ is odd.}\]
  Set $n = k_\phi$ and take $i_0, \dots i_n < \omega$, all of them odd. By the Baldwin-Saxl lemma we may assume
  \[ H_{i_0} \cap \dots H_{i_n} = H_{i_1} \cap \dots H_{i_n}. \]
  By indiscernibility this implies
  \[ H_{i_0+1} \cap \dots H_{i_n} = H_{i_1} \cap \dots H_{i_n}. \]
  But this is a contradiction since $a \not \in H_{i_0+1}$.
\end{proof}

While it is clearly sufficient to assume that the $\phi$-definable subgroups normalize each other, the above proof requires some normality assumption.

\begin{question}
Does \cref{ntp2_nip_transfer} hold even without any normality assumption?
\end{question}

\bibliographystyle{plain}
\bibliography{bibliography}

\end{document}